\documentclass{amsart}
\usepackage{float}
\usepackage{longtable}
\usepackage{graphicx}
\usepackage{pdflscape}
\usepackage[numbers,sort&compress]{natbib}

\usepackage{psfrag}

\usepackage[utf8]{inputenc} 
\usepackage{tgpagella}
\usepackage
[
    a4paper,
    left=1.25in,
    right=1.25in,
    top=1.25in,
    bottom=1.25in
]{geometry}
\usepackage{amsmath,amsthm,amssymb}
\usepackage{mathrsfs}
\usepackage{dsfont}

\usepackage{amsfonts,amscd,xspace,mathtools,faktor,tikz-cd,stmaryrd}
\usepackage{bbm,xfrac}
\usepackage{graphicx}
\usepackage{pdfpages}
\usepackage{enumerate}
\usepackage{tikz}
\usetikzlibrary{knots}

\usepackage{diagbox}
\usepackage{makecell}

\usepackage[numbers]{natbib}         
\usepackage[colorlinks]{hyperref}    

\usepackage{color}   
\usepackage{hyperref}
\hypersetup{
    colorlinks=true, 
    linktoc=all,     
    linkcolor=magenta,      
    urlcolor=teal,
    citecolor=green
}
\usepackage{listings}
\lstdefinelanguage{GAP}{%
  morekeywords={%
    Assert,Info,IsBound,QUIT,%
    TryNextMethod,Unbind,and,break,%
    continue,do,elif,%
    else,end,false,fi,for,%
    function,if,in,local,%
    mod,not,od,or,%
    quit,rec,repeat,return,%
    then,true,until,while%
  },%
  sensitive,%
  morecomment=[l]\#,%
  morestring=[b]",%
  morestring=[b]',%
}[keywords,comments,strings]

\usepackage[T1]{fontenc}
\usepackage[variablett]{lmodern}
\usepackage{xcolor}
\newcommand{\N}{\mathbb{N}}
\newcommand{\Z}{\mathbb{Z}}
\newcommand{\Q}{\mathbb{Q}}

\newcommand{\F}{\mathbb{F}}
\newcommand{\C}{\mathbb{C}}

\newcommand{\CP}{\mathbb{C}P}

\newcommand{\ii}{\mathbf{i}}
\newcommand{\jj}{\mathbf{j}}
\newcommand{\Fl}{\mathrm{Fl}}
\newcommand{\UFl}{\overline{\mathrm{Fl}}}

\DeclarePairedDelimiterX{\inp}[2]{\langle}{\rangle}{#1, #2}

\makeatletter
\newcommand*{\doublerightarrow}[2]{\mathrel{
  \settowidth{\@tempdima}{$\scriptstyle#1$}
  \settowidth{\@tempdimb}{$\scriptstyle#2$}
  \ifdim\@tempdimb>\@tempdima \@tempdima=\@tempdimb\fi
  \mathop{\vcenter{
    \offinterlineskip\ialign{\hbox to\dimexpr\@tempdima+2em{##}\cr
    \rightarrowfill\cr\noalign{\kern.3ex}
    \rightarrowfill\cr}}}\limits^{\!#1}_{\!#2}}}
\newcommand*{\triplerightarrow}[1]{\mathrel{
  \settowidth{\@tempdima}{$\scriptstyle#1$}
  \mathop{\vcenter{
    \offinterlineskip\ialign{\hbox to\dimexpr\@tempdima+2em{##}\cr
    \rightarrowfill\cr\noalign{\kern.2ex}
    \rightarrowfill\cr\noalign{\kern.2ex}
    \rightarrowfill\cr}}}\limits^{\!#1}}}
\newcommand*{\quadrightarrow}[1]{\mathrel{
  \settowidth{\@tempdima}{$\scriptstyle#1$}
  \mathop{\vcenter{
    \offinterlineskip\ialign{\hbox to\dimexpr\@tempdima+2em{##}\cr
    \rightarrowfill\cr\noalign{\kern.1ex}
    \rightarrowfill\cr\noalign{\kern.1ex}
    \rightarrowfill\cr\noalign{\kern.1ex}
    \rightarrowfill\cr}}}\limits^{\!#1}}}    
\makeatother

\let\oldemptyset\emptyset
\let\emptyset\varnothing

\usepackage{colortbl}					
\providecommand{\customgenericname}{}
\newcommand{\newcustomtheorem}[2]{%
  \newenvironment{#1}[1]
  {%
   \renewcommand\customgenericname{#2}%
   \renewcommand\theinnercustomgeneric{##1}%
   \innercustomgeneric
  }
  {\endinnercustomgeneric}
}

\theoremstyle{plain}
\newtheorem{theo}{Theorem}[section]
\newtheorem{cor}[theo]{Corollary}
\newtheorem{lema}[theo]{Lemma}
\newtheorem{prop}[theo]{Proposition}
\newtheorem{innercustomgeneric}{\customgenericname}
\newcustomtheorem{customthm}{Theorem}

\theoremstyle{definition}
\newtheorem{dfn}[theo]{Definition}
\newtheorem{exm}[theo]{Example}

\theoremstyle{remark}

\usepackage[centerlast,small,sc]{caption}

\begin{document}

\author[S. Jana]{Santanil Jana}
\address{Mathematics Department, 
	University of British Columbia}
\email{santanil@math.ubc.ca}

\title[Classifying Space for Commutativity in \texorpdfstring{$U(3)$}{U(3)}]{On the Cohomology of the Total Space of Classifying Space for Commutativity in \texorpdfstring{$U(3)$}{U(3)}}
\maketitle

\begin{abstract}
    The classifying space for commutativity for a Lie group $G$, denoted as $B_{com} G$ was first introduced in \cite{Adem-Cohen-Torres}. The total space of the principal $G$-bundle associated with $B_{com} G$ is denoted by $E_{com} G$. In this paper, we describe $E_{com} U(3)$ as a homotopy colimit of a diagram of spaces and offer a computation of the mod $2$ and mod $3$ cohomologies of $E_{com} U(3)$ by utilizing the spectral sequence associated with a homotopy colimit. We investigate the cohomology of different spaces in the homotopy colimit diagram. These spaces are intriguing in their own right and contribute to the overall fascination of the analysis. We also present the ring structure of the rational cohomology of $E_{com} U(3)$.
\end{abstract}


\section{Introduction}

Classifying spaces are fundamental objects in algebraic topology, as they allow for the classification of principal bundles. These spaces also have significant applications in gauge theories within the standard model of particle physics. In 2012, Adem-Cohen-Torres \cite{Adem-Cohen-Torres} introduced the concept of the classifying space for commutativity in Lie groups, which was motivated by questions that arose within the context of gauge theory \cite{kac}. The spaces of homomorphism and classifying space for commutativity have been an active area of research in Algebraic Topology since their inception. Recent research has focused on studying the algebraic and geometric properties of these spaces, as well as their applications in different areas of mathematics and physics. For instance, there have been studies on the homology and homotopy type of these spaces, including recent works in \cite{Adem-Gomez}, \cite{antolin-camarena-villarreal}, \cite{okay20}, and \cite{adem2022}. Additionally, Bergeron-Silberman has explored the connection between these spaces and the representation theory of Lie groups in \cite{bergeron2016}. Recently, the study of commutative and nilpotent K-theory (\cite{adem2017infinite}, \cite{gritschacher2018}) has gained interest, and it has direct connections with the classifying spaces for commutativity. In \cite{Okay2021}, Okay-Sheinbaum constructed the classifying space for quantum contextuality which is a modified version of the original construction with a focus on its applications in contextuality.\\

For a topological group $G$, let $C_n (G)$ be the set of $n$-tuples $(g_1 , \dots ,g_n)$ such that $g_i g_j = g_j g_i$ for all $1 \le i,j \le n$. We topologize $C_n(G)$ as a subspace of $G^n$ and can identify it with the space of homomorphisms $ \mathrm{Hom} (\Z^n, G) $. The spaces $\{ \mathrm{Hom} (\Z^n , G) \}_{n\ge 0}$ can be assembled into a simplicial space $\mathrm{Hom} (\Z^* , G)$, whose geometric realization $|\mathrm{Hom} (\Z^* , G)|$ is denoted as $B_{com} G $ and was first introduced in \cite{Adem-Cohen-Torres}. The universal principal $G$-bundle over $BG$ can be pulled back to a principal $G$-bundle over $B_{com} G$, and we refer to its total space as $E_{com} G$. The space $E_{com} G$ has notable characteristics and can also be described simplicially. In this paper, we will investigate the cohomology of $E_{com} G$ for the case when $G = U(3)$.\\

We assume that $G$ is locally compact, Haussdorff and $1_G \in G$ is a non-degenerate basepoint. We can associate a simplicial space $[B_{com} G]_*$ to $G$ in the following way. For all $n\ge 0$ define  \[ [B_{com} G]_n := \mathrm{Hom} (\Z^n, G) \subset G^n . \] 
Note that $[B_{com} G]_n$ is given the subspace topology. The face and degeneracy maps are given by the restrictions of the corresponding maps in the bar construction $[BG]_*$. \begin{align*}
    \partial_i (g_1, \dots, g_n) &= \begin{cases} (g_2, \dots, g_n) & \text{if } i=0, \\ (g_1, \dots, g_i g_{i+1}, \dots, g_n) & \text{if } 1\le i \le n-1, \\ (g_1, \dots, g_{n-1}) & \text{if } i=n, \\ \end{cases} \\
    \sigma_j (g_1, \dots, g_n) &= (g_1 , \dots , g_j , 1_G , g_{j+1}, \dots , g_n).
\end{align*}
We define $B_{com} G = |[B_{com} G]_*| $, the geometric realization of the simplicial space $[B_{com} G]_*$. Similarly, we can define $[E_{com} G]_n := \mathrm{Hom} (\Z^n, G) \times G \subset G^{n+1}$ and use the analogous face and degeneracy maps to define $E_{com} G = |[E_{com} G]_*| $. In \cite{Adem-Cohen}, Adem-Cohen proved that $[B_{com} G]*$ and $[E_{com} G]_*$ are proper simplicial spaces, for Lie groups $G$ that are closed subgroups of $GL(n,\mathbb{C})$. Later, this result was extended to hold for any Lie group by Adem-G{\'o}mez in \cite{Adem-Gomez}.\\

We can define a simplicial map by projecting the first $n$-coordinates of $[E_{com} G]_n$ to its base space $[B_{com} G]_n$. This gives us a continuous map $p_{com}: E_{com} G \rightarrow B_{com} G$, which in turn defines a principal $G$-bundle. This bundle can be seen as a restriction of the universal $G$-bundle $p: EG \rightarrow BG$, and we have a morphism of principal $G$-bundles \[ \begin{tikzcd} E_{com} G \ar[r] \ar[d, "p_{com}"] & EG \ar[d, "p"] \\ B_{com} G \ar[r, "i"] & BG. \end{tikzcd} \] This gives rise to a fibration sequence $E_{com} G \rightarrow B_{com} G \rightarrow BG$. The following definition is due to Adem-G{\'o}mez \cite{Adem-Gomez}.
\begin{dfn}
    Suppose that $X$ is a CW complex. A principal $G$-bundle $q: E \rightarrow X$ is said to be \textit{transitionally commutative} if there exists an open cover $\{ U_i \}$ of $X$ such that the bundle $q: E \rightarrow X$ is trivial over each $U_i$, and the transition functions $\rho_{i,j}: U_i \cap U_j \rightarrow G$ commute with each other whenever they are defined simultaneously.
\end{dfn}
The space $BG$ classifies all principal $G$-bundles \cite{Steenrod51}. More precisely, given a principal $G$-bundle $q: E \rightarrow X$, there is a continuous map $f:X \rightarrow BG$ up to homotopy such that $q: E\rightarrow X$ is isomorphic to $f^*p: f^* (EG) \rightarrow X$ as principal $G$-bundles. In other words, there is a bijection between $[X, BG]$ and $P_G (X)$. The space $B_{com} G$ classifies all transitionally commutative principal $G$-bundles in the following sense.

\begin{theo} \label{tr-com} (\cite{Adem-Gomez}, Theorem~2.2)
    Suppose we have a principal G-bundle $q: E \rightarrow X$ over a finite CW complex $X$ and let $f: X \rightarrow BG$ be the classifying map of $q$. Then, $f$ factors through $B_{com}G$ up to homotopy if and only if $q$ is transitionally commutative.
\end{theo}

The computations in this paper are inspired by Antol{\'i}n-Gritschacher-Villarreal \cite{antolin-camarena-villarreal} where they computed the cohomology of $B_{com} G$ and $E_{com} G$ for low-dimensional Lie groups such as $O(2)$, $SU(2)$, $U(2)$, and $SO(3)$. They determined the integral cohomology rings, the mod $2$ cohomology rings along with the action of the Steenrod algebra. In all cases, the cohomology rings were either torsion-free or had only $2$-torsion. In \S \ref{section:2.3}, we present our computation of the cohomology of $E_{com} U(3)$, which is the first example of a Lie group of order $3$ in the complex case. We shall see in \S \ref{subsection:2.3.2} that the cohomology of $E_{com} U(3)$ has both $2$ and $3$-torsion.  
\begin{customthm}{4.6 \& 4.7}
    The $\F_2$ and $\F_3$-cohomology of $E_{com} U(3)$ is given by \[ H^d (E_{com} U(3) ; \F_2 ) = \begin{cases} \F_2 & d=0,5,9,12,13,14 \\ \F_2^{\oplus 2} & d=4,8 \\ \F_2^{\oplus 3} & d=6 \\ 0  & \text{otherwise}   \end{cases} \quad H^d (E_{com} U(3) ; \F_3 ) = \begin{cases} \F_3 & d=0,5,7,13 \\ \F_3^{\oplus 2} & d=4,9,12 \\ \F_3^{\oplus 3} & d=6,8,10,11 \\ 0 & \text{otherwise}   \end{cases}  \] 
\end{customthm}
The rational cohomology of $E_{com} U(n)$ was described additively in Corollary 8.2 of \cite{Adem-Gomez}. In \S \ref{section:2.3.3}, we describe the ring structure of $H^* (E_{com} U(3) ; \Q)$.
\section{The Space \texorpdfstring{$E_{com} G$}{EcomG} as a Homotopy Colimit} \label{section:2.1}

Adem-G{\'o}mez gave a description of the spaces $B_{com} G$ and $E_{com} G$ for a compact connected Lie group $G$ as homotopy colimits over a finite poset in \cite{Adem-Gomez}. In our computations in \S \ref{section:2.3}, we will use this description directly. Therefore, we provide a review of some definitions and results from \cite{Adem-Gomez} that are relevant to our purpose.

\begin{dfn}
Let $G$ be a Lie group. Define $\mathcal{T} (G) $ to be the poset whose objects are closed subspaces $S\subset G$, which are intersections of a collection of maximal tori in $G$. The partial order in $\mathcal{T} (G)$ is given by inclusion.   
\end{dfn}
The topology on the set $\mathcal{T} (G)$ can be described as follows. Let $\mathcal{C} (G)$ be the set of all closed subspaces of $G$. Suppose $\mathcal{U} = \{ U_1 ,\dots ,U_k \}$ is a finite collection of open sets in $G$. Define $\mathcal{C} (G,\mathcal{U})$ as the set of elements $A\in \mathcal{C} (G)$ such that $A\subseteq \cup_{i=1}^{k} U_i$ and $A\cap U_i \neq \oldemptyset $ for all $1\leq i \leq k$. The sets $\mathcal{C} (G,\mathcal{U})$ form a basis for a topology in $\mathcal{C} (G)$ called the finite topology. We give $\mathcal{T} (G)$ the subspace topology via the inclusion $\mathcal{T} (G) \subseteq \mathcal{C} (G)$. \\

Let $\mathfrak{g}$ and $\mathfrak{t}$ be the Lie algebras of $G$ and a maximal torus $T \subset G$ respectively. Also, let $\Phi$ be a root system associated to $(\mathfrak{g}, \mathfrak{t})$ with $\Phi^{+}$ the subset of positive roots. Define for each $\alpha \in \Phi^{+}$ 
\[ \mathfrak{t}_{\alpha} = \{ X\in \mathfrak{t} \mid \alpha (X) \in 2\pi i \Z \}. \]
Let $I= \{ \alpha_1, \dots ,\alpha_k \} $ be a set of positive roots and \[ \mathfrak{t}_I := \cap_{i=1}^{k} \mathfrak{t}_{\alpha_i} \quad \text{and} \quad T_I := \mathrm{exp} (\mathfrak{t}_I ) . \]
We also set $\mathfrak{t}_{\oldemptyset} := \mathfrak{t}$ and hence $T_{\oldemptyset} = \mathrm{exp} (\mathfrak{t} ) = T$. We take a set of simple roots $\Delta = \{\alpha_1,\dots , \alpha_l \}$ for the root system $\Phi$. Denote the Weyl group of $G$ by $W$, which is a reflection subgroup generated by the reflections $s_{\alpha}$, $\alpha\in \Delta$. For a subset $I\subset \Delta$, define $W_I$ to be the subgroup of $W$ generated by the reflections $s_{\alpha}$, $\alpha \in I$. For $I,J \subseteq \Delta$, $W_I$ and $W_J$ are conjugate if and only if $I=wJ$ for some $w\in W$. When this happens we say that $I$ and $J$ are in the same Coxeter class and denote it as $I\sim_W J$. The relation $\sim_W$ is an equivalence relation on the subsets of $\Delta$ and the set of equivalence classes is denoted as $\mathcal{E}_W$.
\begin{theo} \label{top-poset} (\cite{Adem-Gomez}, Theorem 5.4)
Let $G$ be a compact connected Lie group. Fix a set of simple roots $\Delta$. Then any element $S\in \mathcal{T} (G)$ is conjugated to some $T_I$ for $I\subseteq \Delta$. Moreover, there is a $G$-equivariant homeomorphism $\mathcal{T} (G) \cong  \coprod_{[I] \in \mathcal{E}_W} G/N_G (T_I) $
\end{theo}
Using the homeomorphism from Theorem~\ref{top-poset}, we demonstrate the topological poset $\mathcal{T} (U(n))$ for $n\geq 1$ in the following example. 
\begin{exm} \label{exTU(n)}(\cite{Adem-Gomez}, Example~5.5)
     A partition of $n$ is a non-decreasing sequence of integers $\lambda = (\lambda_1, \dots , \lambda_k)$ such that $n = \lambda_1 + \cdots + \lambda_k$. We write $\lambda \vdash n$ to mean that $\lambda$ is a partition of $n$. Given a partition $\lambda \vdash n$, we define the flag manifold $\Fl_{\lambda} (\C) $ as the quotient space $U(n)/(U(\lambda_1)\times \cdots \times U(\lambda_k))$. Geometrically, $\Fl_{\lambda} (\C)$ consists of all possible nested chains of complex subspaces $V_1 \subset V_2 \subset \cdots \subset V_k =\C^n$ of $\mathbb{C}^n$, where each subspace $V_i$ has dimension $\lambda_1 + \cdots + \lambda_i$. We can view $\lambda$ as an ordered $k$-tuple $(\lambda_1, \dots , \lambda_k)$ and the symmetric group $\Sigma_k$ acts on the set of such $k$-tuples by permutation. The isotropy subgroup of $\Sigma_k$ at $\lambda$ under this action is denoted by $\Sigma_{\lambda}$. Then, the topological poset $\mathcal{T} (U(n)) $ is homeomorphic to $ \coprod_{\lambda \vdash n} \Fl_{\lambda} (\C) /\Sigma_{\lambda}$.
\end{exm}

For a compact connected Lie group $G$, let $n= \mathrm{rank} (G) - \mathrm{rank} (Z)$, where $Z:=Z(G)$ is the center of $G$. Define $\rho : \mathcal{T} (G) \rightarrow \N \cup \{ 0\}$ by $\rho (S) := \mathrm{rank} (S) - \mathrm{rank} (Z)$. We see that for all $S\in \mathcal{T} (G)$, $0\le \rho (S) \le n$. Also, $\rho$ is strictly increasing and constant on each connected component of $\mathcal{T} (G)$. Let $\mathcal{S} (n)$ be the poset of all non-empty subsets of $\{ 0,1,\dots , n\}$, with the order being the reverse inclusion of sets. Elements of $\mathcal{S} (n)$ are of the form $\ii =\{ i_0 , \dots ,i_k \} $, where $0\le i_0 < i_1 < \cdots < i_k \le n $. We define the functors $\mathcal{F}_G$ and $\mathcal{H}_G$ associated with a compact connected Lie group $G$ as follows. Let $\ii = \{ i_0 , \dots ,i_k \} \in \mathcal{S} (n)$ and define  
\begin{align*}
    \mathcal{H}_G (\ii) &:= \{ (S_0 , \dots , S_k,b )| S_0 \subset \cdots \subset S_k \in \mathcal{T} (G), \text{ } \rho (S_r) = i_r \text{ for } 0\le r \le k \text{ and } b\in G/S_0 \}.
\end{align*}
Note that $\mathcal{H}_G (\ii) \subset \mathcal{T} (G)^{k+1} \times G/Z$. We give $\mathcal{H}_G (\ii)$ the subspace topology. If $\jj \subset \ii$, then the natural projection map induce continuous functions $p_{\ii , \jj} : \mathcal{H}_G (\ii) \rightarrow \mathcal{H}_G (\jj)$. Hence, $\mathcal{H}_G : \mathcal{S} (n) \rightarrow \mathbf{Top}$ defines a functor. Fix $\ii = \{ i_0, \dots ,i_k \} \in \mathcal{S} (n)$. The space $\mathcal{H}_G (\ii)$ can be described explicitly in the following way. Let $\mathbf{S}_{\ii} = \{ S_{i_0} , \dots ,S_{i_k} \}$ and $\mathbf{S'}_{\ii} = \{ S'_{i_0} , \dots ,S'_{i_k} \}$ be two chains in $\mathcal{T} (G)$ such that $\rho (S_{i_l}) = \rho (S'_{i_l}) = i_l$ for $0\le l \le k$. Then we define the equivalence relation on the chains in $\mathcal{T} (G)$ by setting $\mathbf{S_{\ii}} \sim \mathbf{S'}_{\ii} $ if and only if there is some $g\in G$ such that $g\mathbf{S_{\ii}} g^{-1} = \mathbf{S'_{\ii}}$, that is $g S_{i_l} g^{-1} = S'_{i_l}$ for all $0\le l \le k$. Denote the set of equivalence classes by $\mathcal{E} (\ii)$. Then we have a continuous map
\begin{align*}
    \mu_{\mathbf{S}_{\ii}} : G \times G/S_0 &\longrightarrow \mathcal{H}_G (\ii) \\
    (g,a) &\longmapsto (g\mathbf{S}_{\ii} g^{-1} , gag^{-1})
\end{align*}
Let $ N(\mathbf{S}_{\ii}) := N_G (\mathbf{S}_{\ii})$ be the normalizer of $\mathbf{S}_{\ii}$ in $G$ \textit{i.e.} the subgroup of $G$ consisting of all $g\in G$ such that $g \mathbf{S}_{\ii} g^{-1} = \mathbf{S}_{\ii}$. The groups $N (\mathbf{S}_{\ii})$ acts on $G$ by $n \cdot g = gn^{-1}$ and on $G/S_{0}$ by conjugation. The map $\mu_{\mathbf{S}_{\ii}}$ is invariant under the diagonal action of $N (\mathbf{S}_{\ii})$ on $G\times G/S_0$. If we vary $\mathbf{S}_{\ii}$ through all the equivalence classes in $\mathcal{E} (\ii)$ we get a continuous map \[  \mu_{\ii} := \coprod_{[\mathbf{S}_{\ii}]\in \mathcal{E} (\ii)} \mu_{\mathbf{S}_{\ii}} : \coprod_{[\mathbf{S}_{\ii}]\in \mathcal{E} (\ii)} G \times_{N (\mathbf{S}_{\ii})} G/S_0 \longrightarrow \mathcal{H}_G (\ii).  \]
The mapping $\mu_{\ii}^{-1}$ is continuous and bijective, allowing us to describe $\mathcal{H}_G (\ii)$ using a homeomorphism, similar to the one in \S 6 of \cite{Adem-Gomez} that describes $\mathcal{F}_G (\ii)$.
\begin{equation} \label{HG} 
    \mathcal{H}_G (\ii) \cong \coprod_{[\mathbf{S_{\ii}}] \in \mathcal{E} (\ii)} G\times_{N (\mathbf{S_{\ii}})} G/S_{i_0} .
\end{equation}

The functor $\mathcal{H}_G: \mathcal{S}(2)\rightarrow\mathbf{Top}$ can be interpreted as a diagram of spaces. We can obtain the homotopy colimit over the poset by taking the homotopy colimit of this diagram. Specifically, for the poset $\mathcal{S}(2)=\{ (0), (1), (2), (0,1), (0,2), (1,2), (0,1,2) \}$, we obtain diagram in Figure~\ref{fig:poset2}. We are omitting the set notation while referring to the elements of the poset. So, $\mathcal{H}_G (i,j)$ should be interpreted as $\mathcal{H}_G (\{ i , j\} )$. The homotopy colimit of the diagram in Figure~\ref{fig:poset2} can also be interpreted as a homotopy pushout cube as the one in Figure~\ref{pushoutcube}.
\begin{figure}[t]
    \centering
    \begin{tikzcd}
        & \mathcal{H}_G (1,2) \ar[dl] \ar[dr] \\ 
        \mathcal{H}_G (2) & \mathcal{H}_G (0,1,2) \ar[d] \ar[u] \ar[r] \ar[l] \ar[dr] \ar[dl] & \mathcal{H}_G (1)  \\ 
        \mathcal{H}_G (0,2) \ar[u] \ar[r]  & \mathcal{H}_G (0) & \mathcal{H}_G (0,1) \ar[u] \ar[l] 
\end{tikzcd}
    \caption{The Diagram of Spaces $\mathcal{H}_{G} (\ii)$ for $ \ii \in \mathcal{S} (2)$.}
    \label{fig:poset2}
\end{figure} 

\begin{figure}[H]
    \centering
    \begin{tikzcd}
    & \mathcal{H}_G (1,2) \ar[rr] \ar[dd, dashed] && \mathcal{H}_G (2) \ar[dd] \\ 
    \mathcal{H}_G (0,1,2) \ar[dd] \ar[rr] \ar[ur] && \mathcal{H}_G (0,2) \ar[ur] \ar[dd] \\ 
    & \mathcal{H}_G (1) \ar[rr, dashed] &&  E_{com} U(3) \\
    \mathcal{H}_G (0,1) \ar[rr] \ar[ur] && \mathcal{H}_G (0) \ar[ur]
\end{tikzcd}
    \caption{Homotopy Pushout Interpretation of $\mathcal{S} (2)$.}
    \label{pushoutcube}
\end{figure} 

\begin{theo} \label{adem-gomez-hocolim} (\cite{Adem-Gomez}, Theorem 6.5)
    Suppose that $G$ is a compact connected Lie group. Then there are natural homotopy equivalences \[ \mathrm{hocolim}_{\ii \in \mathcal{S} (n)} \mathcal{H}_G (\ii) \simeq E_{com} G_{\mathds{1}}. \]
\end{theo}

The sets $\mathcal{E} (\ii)$ that appear in the previous description were expressed in terms of the root
system $\Phi$ associated with a maximal torus $T \subset G$ in \cite{Adem-Gomez}. Let $i = \{ i_0, \dots , i_k \} \in \mathcal{S} (n)$ and $\mathbf{S}_{\ii} = \{ S_0 \subset \cdots \subset S_k \}$ be a chain in $\mathcal{T} (G)$ with $\rho (S_r) = i_r$ for $0 \le r \le k$. Any such chain is conjugated to a chain of the form $T_{I_0} \subset T_{I_1} \subset \cdots \subset T_{I_k}$ for a collection of sets of roots $I_0,\dots, I_k$. Furthermore, two chains of this form, $T_{I_0} \subset T_{I_1} \subset \cdots \subset T_{I_k}$ and $T_{J_0} \subset T_{J_1} \subset \cdots \subset T_{J_k}$, are conjugated if and only if there exists $w\in W$ such that $T_{I_r} = T_{wJ_r}$ for $0\le r\le k$. Therefore, the set $\mathcal{E} (\ii)$ can be identified with the set of equivalence classes of sequences of sets of roots of the form $(I_0,\dots,I_k)$ with $\rho(T_{I_r}) = i_r$ for $0\le r\le k$, where $(I_0,\dots, I_k) \sim (J_0,\ldots, J_k)$ if and only if there exists $w\in W$ such that $T_{I_r} =T_{wJ_r}$ for $0\le r\le k$. \\

Our focus in this paper is primarily on the scenarios where the Lie group $G$ is either $U(n)$ or $SU(n)$. For such cases, $B_{com} G$ and $E_{com} G$ are path-connected, since $\mathrm{Hom} (\Z^n , G)$ is path-connected. Additionally, $E_{com} G$ is $3$-connected (\cite{Adem-Gomez}, Proposition 3.3). The computations for $G=U(2)$ and $SU(2)$ were carried out in \cite{antolin-camarena-villarreal}. In the remaining part of this section, we delve into the specifics of the $G=U(3)$ case and make use of the spectral sequence associated with a homotopy colimit (\cite{bk72}, \S XI.7) to determine the cohomology of $E_{com} U(3)$. Before going into the cohomology computations we state a general result relating $E_{com} U(n)$ and $E_{com} SU(n)$.

\begin{lema} \label{grit1.2.8} (\cite{gritschacher}, Lemma 1.2.8) If $\widetilde{G} \rightarrow G$ is a covering homomorphism of compact connected lie groups, then the following diagram is a homotopy pullback diagram.
\[ \begin{tikzcd}   
    B_{com} \widetilde{G}_{\mathds{1}} \ar[r] \ar[d] & B\widetilde{G} \ar[d] \\
    B_{com} G_{\mathds{1}} \ar[r] & BG
\end{tikzcd} \]
\end{lema}

\begin{prop} \label{un=sun}
There are homotopy equivalences $E_{com} U(n) \simeq E_{com} SU(n)$ for all $n\ge 2$.
\end{prop}
\begin{proof}
    We have a covering homomorphism $S^1 \times SU(n) \longrightarrow U(n)$ given by $(\lambda, g) \longmapsto \lambda g$. Applying the above lemma gives us the following homotopy pullback square:
    \[ \begin{tikzcd}   
    B S^1 \times B_{com} SU(n) \ar[r] \ar[d] & BS^1 \times BSU(n) \ar[d] \\
    B_{com} U(n) \ar[r] & BU(n)
\end{tikzcd} \]
The homotopy fibers of the horizontal maps are homotopy equivalent and hence $$E_{com} U(n) \simeq ES^1 \times E_{com} SU(n) \simeq E_{com} SU(n). $$
\end{proof} 

By computing the $H^* (E_{com} U(n))$, we also determine $H^* (E_{com} SU(n))$ as well by Proposition~\ref{un=sun}. However, it is important to note that the cohomology of $B_{com} U(n)$ differs from that of $B_{com} SU(n)$. The case $n=2$ was computed by Antol{\'i}n-Gritschacher-Villarreal \cite{antolin-camarena-villarreal}.
\section{Cohomology of The Spaces in the Homotopy Colimit} \label{section:2.2}


In this section, we focus on the cohomology of the spaces involved in the homotopy colimit described in Theorem~\ref{adem-gomez-hocolim} for $E_{com} U(3)$. First let us describe the spaces $\mathcal{H}_{U(3)} (\ii)$ for $ \ii \in \mathcal{S} (2)$. Recall from Example~\ref{exTU(n)} that \[ \mathcal{T} (U(3)) = (\Fl_{\lambda^1} (\C)/\Sigma_{\lambda^1}) \sqcup (\Fl_{\lambda^2} (\C)/\Sigma_{\lambda^2}) \sqcup  (\Fl_{\lambda^3} (\C) / \Sigma_{\lambda^3}) , \] where $\lambda^1 = (1,1,1)$, $ \lambda^2 = (1,2)$, and $ \lambda^3 = (3)$. From Example~\ref{exTU(n)}, $\Delta = \{ \alpha_1, \alpha_2 \}$ is a set of simple roots for the root system associated with the maximal torus $T(3) \subset U(3)$. The subsets of $\Delta$ are $\oldemptyset, \{\alpha_1\}, \{\alpha_2\}, $ and $\Delta$. From our discussions in Example~\ref{exTU(n)}, $T_{\oldemptyset}$ is conjugated to $T_{I(\lambda^1)}$, $T_{\{\alpha_1\}}$ and $T_{\{\alpha_2\}}$ are conjugated to $T_{I(\lambda^2)}$ and $T_{\Delta}$ is conjugated to $T_{I(\lambda^3)}$. Let us denote by $T(1)$ and $T(2)$ the tori \[ T(1):= \{ \mathrm{diag} (x_1, x_1 ,x_1) \mid x_1 \in S^1 \} \quad \text{and} \quad T(2):= \{ \mathrm{diag} (x_1, x_1 ,x_2) \mid x_i \in S^1 \} \] in $U(3)$ and by $N(d)$ the normalizer $N_{U(3)} (T(d))$. Then $\mathcal{T} (U(3))$ is given by \[ \mathcal{T} (U(3)) \cong \{ * \} \sqcup U(3)/N(2) \sqcup U(3)/N(3).  \] As, $T_{\{ \alpha_1 \}} = T_{(12)\{ \alpha_2\}}$ from our previous description of $\mathcal{E} (\ii)$ as equivalence classes of sequences of sets of roots, we see that $\mathcal{E} (\ii)$ consists of a single equivalence class for all $\ii \in \mathcal{S} (2)$. \\


The normalizer of a chain $T_{I_0} \subset \cdots \subset T_{I_k}$ in $U(3)$ is $N (T_{I_k})$ as $N (T_{I_r}) \subseteq N (T_{I_{r+1}} )$. We can now describe all the spaces $\mathcal{H}_{U(3)} (\ii)$ for $\ii \in \mathcal{S} (2)$. From now on we will denote $\mathcal{H}_{U(3)}$ by $\mathcal{H}$. We have \begin{align*}
    \mathcal{H} (0) &= U(3) \times_{N(1)} U(3)/T(1) \cong PU(3), \\
    \mathcal{H} (1) &= U(3) \times_{N(2)} U(3)/T(2) \cong  U(3)/N(2) \times U(3)/T(2), \\
    \mathcal{H} (2) &= U(3) \times_{N(3)} U(3)/T(3) \cong U(3)/T(3) \times_{\Sigma_3} U(3)/T(3), \\
    \mathcal{H} (0,1) &= U(3) \times_{N(2)} U(3)/T(1) \cong U(3)/N(2) \times PU(3), \\
    \mathcal{H} (0,2) &= U(3) \times_{N(3)} U(3)/T(1) \cong U(3)/N(3) \times PU(3), \\
    \mathcal{H} (1,2) &= U(3) \times_{N(3)} U(3)/T(2) \cong U(3)/T(3) \times_{\Sigma_3} U(3)/T(2), \\
    \mathcal{H} (0,1,2) &= U(3) \times_{N(3)} U(3)/T(1) \cong U(3)/N(3) \times PU(3) . 
\end{align*} 

\subsection{Cohomology of \texorpdfstring{$U(3)/N(3)$}{U(3)/N(3)}} \label{subsection:2.2.1}

In order to compute the cohomology of $U(3)/N(3)$, we use the Serre spectral sequence associated with the fiber sequence 
\begin{equation} \label{ssorder3}
    U(3)/T(3) \longrightarrow U(3)/N(3) \longrightarrow B\Sigma_3. 
\end{equation} 
The $E_2$ page of this spectral sequence is given by $E_2^{p,q} = H^p (B\Sigma_3 ; H^q (U(3)/T(3))$ and the spectral sequence converges to $H^{*} (U(3)/N(3))$. It's worth noting that $U(3)/T(3)$ can be seen as the complete complex flag manifold of order $3$, which is denoted by $\Fl_3 (\C)$. The symmetric group $\Sigma_3$ acts on $\Fl_3 (\C)$ naturally by permuting the flags. Consequently, $U(3)/N(3)$ is identified as the unordered flag manifold of order $3$, denoted by $\UFl_3 (\C)$. We will use this identification going forward. The following result is due to Borel. 

\begin{theo} (\cite{fulton1996}) \label{cohoflagcom}
Let $H^* (BT(n) ; \Z) = \Z[x_1, \dots , x_n] $ with $|x_i| = 2$. Then
the cohomology ring $H^*(\Fl_n (\C) ; \mathbb{Z})$  is given by
\[H^*( \Fl_n (\C) ; \mathbb{Z})\cong \frac{\mathbb{Z}[x_1, x_2, \dots, x_n]}{(\sigma_1, \sigma_2, \dots, \sigma_n)}  \]
where $\sigma_i(x_1, x_2, \dots, x_n)$ are elementary symmmetric functions. 
\end{theo}
By the above theorem, $H^* (\Fl_3 (\C); \mathbb{Z} )$ is torsion free, of total rank $|G| = 6$, with even dimensional cohomology concentrated in degrees $0,2,4,6$. To determine $E_2^{p,q}$, we describe $H^d ( \Fl_3 (\C) ; \mathbb{Z})$ as integral representations $M_d$ of $\Sigma_3$ for $d=0,2,4,6$ and compute the group cohomology $H^* (\Sigma_3; M_d)$. The action of $\Sigma_3$ on $H^d (\Fl_3 (\C); \mathbb{Z})$ is via
\begin{itemize}
    \item the trivial representation $\Z$ in degree $d=0$,
    \item the sign representation $S$ in degree $d=6$,
    \item the standard representation $M$ and $M'$ in degrees $d=2$ and $d=4$ respectively.
\end{itemize}
Also, note that $\{ x_1^{i_1} \cdot x_2^{i_2} \cdot x_3^{i_3} | i_j \le 3-j \} $ forms a basis for $H^* (\Fl_3 (\C) ; \mathbb{Z})$ and hence we can write
\begin{align*}
    H^0 (\Fl_3 (\C) ; \mathbb{Z}) &= \mathbb{Z},  \\
    H^2 (\Fl_3 (\C) ; \mathbb{Z}) &= \frac{ \mathbb{Z} \{ x_1 ,x_2 ,x_3 \} }{(x_1+x_2+x_3)} =: M, \\
    H^4 (\Fl_3 (\C) ; \mathbb{Z}) &= \frac{ \mathbb{Z} \{ x_1 x_2 , x_2 x_3 ,x_3 x_1 \}}{(x_1 x_2 + x_2 x_3 +x_1 x_3)}  =: M' \cong M ,\\
    H^6 (\Fl_3 (\C) ; \mathbb{Z} ) &=  \mathbb{Z} \{x_1^2 x_2 \} := S.  
\end{align*}
As representations of $\Sigma_3$, the modules $M$ and $ M'$ are isomorphic as they are both quotients of the standard rank three permutation module by its rank one submodule of invariants. More precisely, $M \cong M'$ can be described via the short exact sequence \begin{equation} \label{eq:M}
    0 \longrightarrow \mathbb{Z} \longrightarrow \mathbb{Z} [C_3] \longrightarrow M \longrightarrow 0 
\end{equation} 
where $C_3 \cong A_3 \subset \Sigma_3$ is isomorphic to the alternating group.
Similarly, the sign representation $S$ can be described via the short exact sequence \begin{equation} \label{eq:S}
    0 \longrightarrow \mathbb{Z} \longrightarrow \mathbb{Z} [C_2] \longrightarrow S \longrightarrow 0 
\end{equation} where $C_2 := \Sigma_3 / C_3$ is isomorphic to $\mathbb{Z}/2$. Also, recall that the group cohomology of $\Sigma_3$ is given by \[ H^d (\Sigma_3 ; \Z) = \begin{cases} \mathbb{Z}/2 & d = 4k+2 \\ \mathbb{Z}/6 & d=4k \\ 0 & \text{otherwise} \end{cases} \]
From the long exact sequences in group cohomology associated with the above short exact sequences we obtain \begin{equation} \label{3MS}
    H^d (\Sigma_3; M)_3 = \begin{cases}
    \mathbb{Z}/3 & d=4k+3 \\ 0 & \text{otherwise}
\end{cases} \quad  \quad  H^d (\Sigma_3 ; S)_3 = \begin{cases}
    \mathbb{Z}/3 & d=4k+2 \\ 0 & \text{otherwise}
\end{cases}
\end{equation} 
Now let us consider the Serre spectral sequence with the $3$-adic coefficients $\mathbb{Z}_3$ associated to (\ref{ssorder3}). We know that $ E_2^{0,0} \cong \mathbb{Z}_3$ and $E_2^{0,j} = H^j (\Fl_3 (\mathbb{C}) , \mathbb{Z}_3)^{\Sigma_3} = 0$ for all $j>0$. Moreover, for the rows in the spectral sequence, we have for all $i\ge 0$ \begin{align*}
    E_2^{4(i+1), 0} \cong \mathbb{Z}/3 , &\quad E_2^{4i+3, 2} \cong  \mathbb{Z}/3 \\
     E_2^{4i+3, 4} \cong \mathbb{Z}/3, &\quad  E_2^{4i+2, 6} \cong \mathbb{Z}/3.
\end{align*}
All other terms on the $ E_2$ page are zero (see Figure~\ref{fig:ss3U3E2}).
\begin{figure}
    \centering
    \begin{tikzpicture}
    \matrix (m) [matrix of math nodes,
    nodes in empty cells,nodes={minimum width=5ex,
    minimum height=5ex,outer sep=-5pt},
    column sep=.5ex,row sep=2ex]{
                &      &     &     & \\
          6     &  0 &  0  & \mathbb{Z}/3 & 0 & 0 & 0 & \mathbb{Z}/3 & 0 & 0 & \cdots \\
          4     &  0  & 0  & 0  & \mathbb{Z}/3 & 0 & 0 & 0 & \mathbb{Z}/3 & 0 &\cdots  \\
          2     &  0  & 0  & 0  & \mathbb{Z}/3 & 0 & 0 & 0 & \mathbb{Z}/3 & 0 &\cdots  \\
          0     &  \mathbb{Z}_3 &  0  & 0 & 0 & \mathbb{Z}/3 & 0 & 0 & 0 & \mathbb{Z}/3 &\cdots \\
    \quad\strut &   0  &  1  &  2  & 3 & 4 & 5 & 6 & 7 & 8 & \strut \\};
    \draw[-stealth]  (m-2-4.south east) -- (m-4-9.north west);
    \draw[-stealth]  (m-3-5.south east) -- (m-5-10.north west);

    \draw[thick] (m-1-1.east) -- (m-6-1.east) ;
    \draw[thick] (m-6-1.north) -- (m-6-11.north) ;
    \end{tikzpicture}
    \caption{The $E_5$-page of the Serre Spectral Sequence associated with (\ref{ssorder3}) with $3$-adic Coefficients.}
    \label{fig:ss3U3E2}
\end{figure}

Considering possible differentials we notice that $d_3 $ and $d_7$ must be zero as $E_3^{4i+5, 4} = E_3^{4i+6,2} = E_3^{4i+6,0} = 0$ and $E_7^{4i+9, 0} = 0$. As the spectral sequence converges to the cohomology of a six-dimensional manifold $\overline{\Fl}_3(\mathbb{C})$, we must have $d_5 \neq 0$ and for all $i\ge 0$ \begin{align*}
    d_5 : E_5^{4i+2, 6} \xrightarrow{\cong} E_5^{4(i+1)+3, 2} \\
    d_5 : E_5^{4i+3, 4} \xrightarrow{\cong} E_5^{4(i+1)+4, 0}
\end{align*}
Hence, the only terms remaining on the $E_{\infty}$ page are $E_{\infty}^{3,2} = E_{\infty}^{4,0} \cong \mathbb{Z}/3$. Therefore we obtain the following theorem.
\begin{theo}
    The $3$-adic cohomology of $\UFl_3 (\C)$ is given by \[ H^d (\UFl_3 (\C); \mathbb{Z}_3) \cong \begin{cases}
        \mathbb{Z}_3 & d = 0 \\ \mathbb{Z}/3 & d=4,5 \\ 0 & \text{otherwise}
    \end{cases} \]
\end{theo}
\begin{cor}
    The Poincar\'e series of the mod $3$ cohomology of $\UFl_3 (\C) $ is given by \[ \Pi_{\UFl_3 (\C)}^3 (t) = 1+t^3 + 2t^4 +t^5  .  \]
\end{cor}

We now focus on the mod 2 case. We have an isomorphism $H^* (\Sigma_3; \mathbb{Z})_2 \cong H^* (C_2 ; \mathbb{Z})_2$ induced by the inclusion $\mathbb{Z} \hookrightarrow \mathbb{Z} [\Sigma_3 / C_2]$. So, we see that $H^d (\Sigma_3; M)_2 = 0$ for all $i\ge 0$. From the long exact sequence in group cohomology associated with the short exact sequence \[  0 \longrightarrow \mathbb{Z} \longrightarrow \mathbb{Z} [\Sigma_3/C_3] \longrightarrow S \longrightarrow 0,  \] 
we deduce that $H^d (\Sigma_3 ; S)_2 \cong H^{d+1} (\Sigma_3 ; \mathbb{Z})_2$. So, \begin{equation} \label{2MS}
    H^d (\Sigma_3 ; S)_2 \cong \begin{cases}
    \mathbb{Z}/2 & d = 2k+1 \\
    0 & \text{otherwise}
\end{cases} 
\end{equation} 
 As before, we consider the same Serre spectral sequence (\ref{ssorder3}), but now with the $2$-adic coefficients $\mathbb{Z}_2$. We have that $E_2^{0,0} \cong \mathbb{Z}_2$ and for all $i \ge 0$ \begin{align*}
     E_2^{2i+1, 6} \cong \mathbb{Z}/2, \quad E_2^{2i+2, 0} \cong \mathbb{Z}/2  
 \end{align*}
and all other terms on the $E_2$ page are zero (see Figure~\ref{fig:ss2U3E2}). Looking at the possible non-zero differentials, the only possibility is \[ d_7: E_7^{2i+1, 6} \to E_7^{2i+8, 0} \] 
which must be an isomorphism for $i\ge 0$, for dimension reasons.

\begin{figure}
    \centering
    \begin{tikzpicture}
  \matrix (m) [matrix of math nodes,
    nodes in empty cells,nodes={minimum width=5ex,
    minimum height=5ex,outer sep=-5pt},
    column sep=.5ex,row sep=2ex]{
                &      &     &     & \\
          6     &  0 &  \mathbb{Z}/2  & 0 & \mathbb{Z}/2 & 0 & \mathbb{Z}/2 & \cdots \\
          \vdots&    &       &   &     &   &     &      \\
          0     &  \mathbb{Z}_2 &  0  & \mathbb{Z}/2 & 0 & \mathbb{Z}/2 & 0 & \mathbb{Z}/2 & 0 & \mathbb{Z}/2 & \cdots \\
    \quad\strut &   0  &  1  &  2  & 3 & 4 & 5 & 6 & 7 & 8 & \strut \\};
    \draw[-stealth]  (m-2-3.south east) -- (m-4-10.north west);
    \draw[thick] (m-1-1.east) -- (m-5-1.east) ;
    \draw[thick] (m-5-1.north) -- (m-5-11.north) ;
    \end{tikzpicture}
    \caption{The $E_7$-page of the Serre Spectral Sequence associated with (\ref{ssorder3}) with $2$-adic Coefficients.}
    \label{fig:ss2U3E2}
\end{figure}
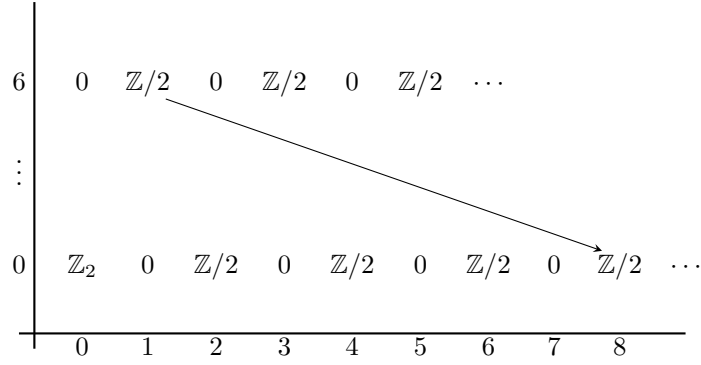

Hence the only surviving terms on the $E_{\infty}$ page are $E_{\infty}^{0,0} = \Z_2 $ and $E_{\infty}^{2,0} = E_{\infty}^{4,0} =E_{\infty}^{6,0} \cong \mathbb{Z}/2$. Therefore, we obtain the following theorem.
\begin{theo} \label{coh3com2}
    The $2$-adic cohomology of $\UFl_3 (\C)$ is given by \[ H^d (\UFl_3 (\C); \mathbb{Z}_2) \cong \begin{cases}
        \mathbb{Z}_2 & d = 0 \\ \mathbb{Z}/2 & d=2,4,6 \\ 0 & \text{otherwise}
    \end{cases} \]
\end{theo}
As a corollary of the above theorem, we can have the following.
\begin{cor} \label{cohoring3}
    Let $\gamma$ denote the non-zero class in $H^1 (\UFl_3 (\C); \mathbb{F}_2) \cong \mathbb{F}_2$. Then \[  H^* (\UFl_3 (\C); \mathbb{F}_2) \cong \frac{\mathbb{F}_2 [\gamma]}{(\gamma^7)} , \] and the mod $2$ Poincar\'e series is given by \[ \Pi_{\UFl_3 (\C)}^2 (t) = 1+t+t^2 +t^3 +t^4 +t^5 +t^6 . \] 
\end{cor}

\subsection{Cohomology of \texorpdfstring{$U(3)/T(3) \times_{\Sigma_3} U(3)/T(3)$}{Lg}} \label{subsection:2.2.2}

We proceed similarly as in the previous subsection. With the identification $U(3)/T(3) \cong \Fl_3 (\C)$ we have the fiber sequence \begin{equation} \label{fl3timesfl3}
    \Fl_3 (\C) \times \Fl_3 (\C) \longrightarrow \Fl_3 (\C) \times_{\Sigma_3} \Fl_3 (\C) \longrightarrow B\Sigma_3 
\end{equation} 
  and in the associated Serre spectral sequence, the $E_2$-page is given by \[ E_{2}^{*,*} \cong H^* (B\Sigma_3 ; H^* (\Fl_3 (\C)^{\times 2} ) ), \]  which converges to $H^* (\Fl_3 (\C) \times_{\Sigma_3} \Fl_3 (\C))$. From the K\"unneth theorem and Theorem~\ref{cohoflagcom}, we have \[ H^* (\Fl_3 (\C)^{\times 2} ; \Z) \cong H^* (\Fl_3 (\C) ; \Z) \otimes H^* (\Fl_3 (\C) ; \Z). \]
Note that $H^d (\Fl_3 (\C)^{\times 2} ; \Z) $ is non-zero if and only if $d$ is even. The cohomology groups of $\Fl_3 (\C)^{\times 2}$ can be expressed in terms of integral representation $\Z$, $M$, and $S$ of $\Sigma_3$ as follows: \begin{align*}
    H^0 (\Fl_3 (\C)^{\times 2} ; \mathbb{Z}) &= \mathbb{Z} , \\
    H^2 (\Fl_3 (\C)^{\times 2} ; \mathbb{Z}) &= M \oplus M \\
    H^4 (\Fl_3 (\C)^{\times 2} ; \mathbb{Z}) &= (M' \oplus M') \oplus M^{\otimes 2} \cong (M \oplus M) \oplus M^{\otimes 2}, \\
    H^6 (\Fl_3 (\C)^{\times 2} ; \mathbb{Z} ) &=  (S \oplus S) \oplus (M \otimes M') \oplus (M' \otimes M) \cong (S \oplus S) \oplus (M^{\otimes 2} \oplus M^{\otimes 2} ), \\
    H^8 (\Fl_3 (\C)^{\times 2} ; \mathbb{Z}) &= (M \otimes S) \oplus (S\otimes M) \oplus {M'}^{\otimes 2} \cong M_S \oplus M_S \oplus M^{\otimes 2}, \\
    H^{10} (\Fl_3 (\C)^{\times 2} ; \mathbb{Z}) &= (M \otimes S) \oplus (S\otimes M) \cong M_S \oplus M_S , \\
    H^{12} (\Fl_3 (\C)^{\times 2} ; \mathbb{Z}) &= S \otimes S \cong \mathbb{Z},
\end{align*}
 where $M_S := M\otimes S$. We will compute the $p$-adic cohomology of $\Fl_3 (\C) \times_{\Sigma_3} \Fl_3 (\C)$ using the Serre spectral sequence associated with (\ref{fl3timesfl3}) for $p=2,3$. From (\ref{3MS}) and (\ref{2MS}), we have the additive descriptions of the group cohomologies $H^* (\Sigma_3; M)_p$ and  $H^* (\Sigma_3; S)_p$. To describe the $E_2$ page of the spectral sequence $E$, we require an additive description of the group cohomologies $H^* (\Sigma_3 ; M^{\otimes 2})$ and $H^* (\Sigma_3 ; M_S)$. As the tensor product is only a right exact functor and we lose exactness on the left, we can not use similar methods used in \S \ref{subsection:2.2.1} and make use of the short exact sequences (\ref{eq:M}) and (\ref{eq:S}). Therefore, we use the HAP (\textit{Homological Algebra Programming} \cite{HAP}) package within the computer algebra system GAP (\textit{Groups, Algorithms, and Programming} \cite{GAP4}) to perform the necessary group cohomology computations. Our results are summarized in the following proposition. For more details on these computations, we refer to Appendix~\ref{appendixA.1}.
\begin{prop} \label{3MMSM}
    We have the following group cohomologies \[ 
    H^d (\Sigma_3; M^{\otimes 2})_3 = \begin{cases}
    \Z_3 & d=0 \\ \mathbb{Z}/3 & d=4k+2 \\ 0 & \text{otherwise}
    \end{cases} \quad  \quad  H^d (\Sigma_3 ; M_S)_3 = \begin{cases}
    \mathbb{Z}/3 & d=4k+1 \\ 0 & \text{otherwise}
\end{cases} \]
\end{prop}
By utilizing our previous results (\ref{3MS}) and Proposition~\ref{3MMSM}, along with the well-known fact from Group cohomology, \[ H^d (G; \oplus_{i=1}^n A_i ) \cong \bigoplus_{i=1}^n H^d (G; A_i), \] we can completely describe the $E_2$-page of the Serre spectral sequence associated with (\ref{fl3timesfl3}) with $3$-adic coefficients. Refer to Figure~\ref{fig:ss3U3U3E2} for the $E_2$-page. The only possible differential which is non-zero is $d_5$. We have illustrated some of the possible non-zero differentials in Figure~\ref{fig:ss3U3U3E2}. Also, we see that there are no other non-zero higher differentials. Hence, the spectral sequence collapses at the $E_6$-page. As, $|\Sigma_3| = 6$, by a transfer argument $H^* (\Fl_3 (\C) \times_{\Sigma_3} \Fl_3 (\C))$ does not have any $9$-torsion and we have the following. 
\begin{theo}
    The $3$-adic cohomology of $\Fl_3 (\C) \times_{\Sigma_3} \Fl_3 (\C)$ is given by \[ H^d (\Fl_3 (\C) \times_{\Sigma_3} \Fl_3 (\C); \mathbb{Z}_3) \cong \begin{cases}
        \Z_3 & d = 0, 12 \\ \Z_3 \oplus \Z/3  & d = 4 \\ {\Z/3}^{\oplus 2} & d=5 \\ \Z_3^{\oplus 2} \oplus \Z/3 & d=6 \\ \Z/3 & d=7,9  \\  Z_3 \oplus {\Z/3}^{\oplus 2} & d=8 \\  0 & \text{otherwise}
    \end{cases} \]
\end{theo}
\begin{cor}
    The Poincar\'e series of the mod $3$ cohomology of $\Fl_3 (\C) \times_{\Sigma_3} \Fl_3 (\C)$ is given by \[ \Pi_{\Fl_3 (\C) \times_{\Sigma_3} \Fl_3 (\C)}^3 (t) = 1+t^3 + 4t^4 +3t^5 +4t^6 + 3t^7 + 4t^4 + t^9 +t^{12} .  \]
\end{cor}

\begin{figure}[t]
    \centering
    \begin{tikzpicture}
    \matrix (m) [matrix of math nodes,
    nodes in empty cells,nodes={minimum width=5ex,
    minimum height=5ex,outer sep=-5pt},
    column sep=.5ex,row sep=2ex]{
                &      &     &     & \\
          12    &   \Z_3 &  0  & 0 & 0 & \Z/3 & 0 & 0 & 0 & \Z/3 & 0 & \cdots \\
          10    &   0  & (\Z/3)^2 & 0 & 0 & 0 & (\Z/3)^2 & 0 & 0 & 0 &  (\Z/3)^2 & \cdots \\
          8     &  \Z_3 & (\Z/3)^2 & \Z/3 & 0 & 0 & (\Z/3)^2 & \Z/3 & 0 & 0 &  (\Z/3)^2 & \cdots \\
          6     &  (\Z_3)^2 &  0  & (\Z/3)^4 & 0 & 0 & 0 & (\Z/3)^4 & 0 & 0 &  0 & \cdots \\
          4     &  \Z_3  & 0  & \Z/3  & (\Z/3)^2 & 0 & 0 & \Z/3 & (\Z/3)^2 & 0 &  0 & \cdots  \\
          2     &  0  & 0  & 0  & (\Z/3)^2 & 0 & 0 & 0 & (\Z/3)^2 & 0 &  0 & \cdots  \\
          0     &  \Z_3 &  0  & 0 & 0 & \Z/3 & 0 & 0 & 0 & \Z/3 &  0 & \cdots \\
    \quad\strut &   0  &  1  &  2  & 3 & 4 & 5 & 6 & 7 & 8 & 9 & \strut \\};
    \draw[-stealth]  (m-6-5.south east) -- (m-8-10.north west);
    \draw[-stealth]  (m-5-4.south east) -- (m-7-9.north west);
    \draw[-stealth]  (m-4-3.south east) -- (m-6-8.north west);
    \draw[-stealth]  (m-4-4.south east) -- (m-6-9.north west);
    \draw[-stealth]  (m-4-7.south east) -- (m-6-12.north west);
    \draw[-stealth]  (m-3-3.south east) -- (m-5-8.north west);
    \draw[-stealth]  (m-3-7.south east) -- (m-5-12.north west);
    \draw[-stealth]  (m-2-2.south east) -- (m-4-7.north west);
    \draw[-stealth]  (m-2-6.south east) -- (m-4-11.north west);
    
    \draw[thick] (m-1-1.east) -- (m-9-1.east) ;
    \draw[thick] (m-9-1.north) -- (m-9-12.north) ;
    \end{tikzpicture}
    \caption{The $E_5$-page of the Serre Spectral Sequence associated with (\ref{fl3timesfl3}) with $3$-adic Coefficients.}
    \label{fig:ss3U3U3E2}
\end{figure}
We now focus on cohomology of $\Fl_3 (\C) \times_{\Sigma_3} \Fl_3 (\C)$ with $2$-adic coefficients. Note that $H^* (\Sigma_3; M_S)_2 =0$ and $H^* (\Sigma_3 ; M^{\otimes 2})_2 \cong \Z_2$ (see Appendix~\ref{appendixA.1}). We can now completely describe the $E_2$-page of the Serre spectral sequence associated with (\ref{fl3timesfl3}) with $2$-adic coefficients. Refer to Figure~\ref{fig:ss2U3U3E2} for the $E_2$-page. It is easy to see that the only possible non-zero differential is $d_7$. Also, all higher differentials are zero and the spectral sequence collapses at the $E_8$-page. Again by a transfer argument, $H^* (\Fl_3 (\C) \times_{\Sigma_3} \Fl_3 (\C))$ does not have any $4$-torsion and hence gives us the following.
\begin{figure}[t]
    \centering
    \begin{tikzpicture}
    \matrix (m) [matrix of math nodes,
    nodes in empty cells,nodes={minimum width=5ex,
    minimum height=5ex,outer sep=-5pt},
    column sep=.5ex,row sep=2ex]{
                &      &     &     & \\
          12    &   \Z_2 &  0  & \Z/2 & 0 & \Z/2 & 0 & 0 & 0 & \Z/2 &  \cdots \\
          8     &  \Z_2  & 0  & 0  & 0 & 0 & 0 & 0 & 0 & 0 &   \cdots  \\
          6     &  (\Z_2)^2 &  (\Z/2)^2  & 0 & (\Z/2)^2 & 0 & (\Z/2)^2 & 0 & (\Z/2)^2 & 0 &  \cdots \\
          4     &  \Z_2  & 0  & 0  & 0 & 0 & 0 & 0 & 0 & 0 &   \cdots  \\
          0     &  \Z_2 &  0  & \Z/2 & 0 & \Z/2 & 0 & \Z/2 & 0 & \Z/2 &   \cdots \\
    \quad\strut &   0  &  1  &  2  & 3 & 4 & 5 & 6 & 7 & 8 &  \strut \\};
    \draw[-stealth]  (m-4-3.south east) -- (m-6-10.north west);
    \draw[-stealth]  (m-2-2.south east) -- (m-4-9.north west);
    
    \draw[thick] (m-1-1.east) -- (m-7-1.east) ;
    \draw[thick] (m-7-1.north) -- (m-7-11.north) ;
    \end{tikzpicture}
    \caption{The $E_7$-page of the Serre Spectral Sequence associated with (\ref{fl3timesfl3}) with $2$-adic Coefficients.}
    \label{fig:ss2U3U3E2}
\end{figure}
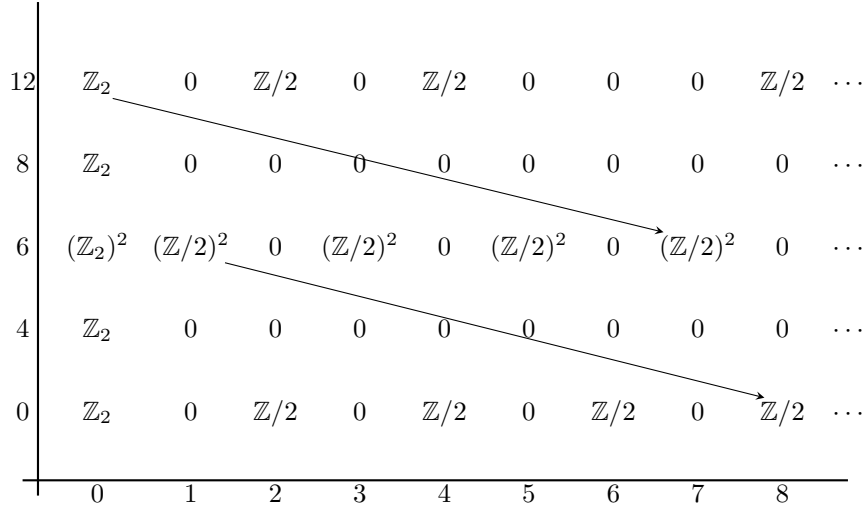

\begin{theo}
    The $2$-adic cohomology of $\Fl_3 (\C) \times_{\Sigma_3} \Fl_3 (\C)$ is given by \[ H^d (\Fl_3 (\C) \times_{\Sigma_3} \Fl_3 (\C); \mathbb{Z}_2) \cong \begin{cases}
        \Z_2 & d = 0, 8, 12 \\ \Z/2 & d=2,7,9,11  \\  \Z_2 \oplus \Z/2  & d = 4 \\  \Z_2^{\oplus 2} \oplus \Z/2 & d=6 \\  0 & \text{otherwise}
    \end{cases} \]
\end{theo}
\begin{cor}
    The Poincar\'e series of the mod $2$ cohomology of $\Fl_3 (\C) \times_{\Sigma_3} \Fl_3 (\C)$ is given by \[ \Pi_{\Fl_3 (\C) \times_{\Sigma_3} \Fl_3 (\C)}^2 (t) = 1+t+t^2 + t^3 + 2t^4 +t^5 +4t^6 + t^7 + 2t^8 + t^9 + t^{10} + t^{11} + t^{12} .  \]
\end{cor}

\subsection{Cohomology of \texorpdfstring{$U(3)/T(2)$}{U(3)/T(2)}}

We will use a similar approach to \S 4 of \cite{BAUM1965} to calculate the mod $2$ and mod $3$ cohomologies of $U(3)/T(2)$. We have the fiber sequence $U(3) \rightarrow U(3)/T(2) \rightarrow (\CP^{\infty})^2$. Since $U(3)/T(2) \rightarrow (\CP^{\infty})^2$ is a principal $U(3)$-bundle, there is a set of generators of $H^* (U(3); \F_p)$ that transgress into Chern classes of the associated complex vector bundle $\eta$ over $(\CP^{\infty})^2$. \\

Recall that $T(2)$ is the 2-torus \[ \{ \mathrm{diag} (\lambda_1 , \lambda_1 , \lambda_2) | \lambda_i \in S^1  \} \cong S^1 \times S^1 .\]  
The inclusion $T(2) \hookrightarrow U(3)$ can be factored by $S^1 \times S^1 \xrightarrow{\Delta} S^1 \times S^1 \times S^1 \xrightarrow{i} U(3)$, where $\Delta (\lambda_1 , \lambda_2) = (\lambda_1 , \lambda_1 , \lambda_2)$. These homomorphisms of groups give us the following maps of classifying spaces \[ BS^1 \times BS^1 \xlongrightarrow{\Delta{'}} BS^1 \times BS^1 \times BS^1 \xlongrightarrow{i'} BU(3).  \] 
Now $i' : (\CP^{\infty})^3 \rightarrow BU(3)$ induces the bundle $p_1^* (\xi) \oplus p_2^* (\xi) \oplus p_3^* (\xi)$, where $\xi$ is the canonical line bundle over $\CP^{\infty}$ and $p_j: (\CP^{\infty})^3 \rightarrow \CP^{\infty}$ is the projection to the $j$-th factor. The total Chern class $c(\xi)$ is given by $1+\alpha$ for $\alpha \in H^2 (\CP^{\infty} ; \F_p)$. By Whitney Product Formula \[ c(p_1^* (\xi) \oplus p_2^* (\xi) \oplus p_3^* (\xi)) = (1+\alpha_1)(1+\alpha_2)(1+\alpha_3), \] where $\alpha_j = p_j^* \alpha$. 
Let $\beta_j \in H^2 ((\CP^{\infty})_j ; \Z)$ for $j=1,2$. Then $(\Delta{'})^{*} (\alpha_1) = \beta_1$, $(\Delta^{'})^{*} (\alpha_2) = \beta_1$, and $(\Delta^{'})^{*} (\alpha_3) = \beta_2$. 
The map $i' \circ \Delta{'}$ induces the bundle $\eta$ over $(\CP^{\infty})^2$. Hence, \begin{align*}
     c(\eta) &= (\Delta^{'})^* ((1+\alpha_1)(1+\alpha_2)(1+\alpha_3)) \\
     &=(1+\beta_1)^2 (1+ \beta_2) \\
     &= 1 + (2\beta_1 + \beta_2) + (\beta_1^2 + 2\beta_1 \beta_2) + \beta_1^2 \beta_2 .
     \end{align*}
So, $c_1 (\eta) = 2\beta_1 + \beta_2$, $c_2 (\eta)=\beta_1^2 + 2\beta_1 \beta_2$, $c_3 (\eta)=\beta_1^2 \beta_2$. Following the result in Theorem~4.1 of \cite{BAUM1965}, we see that the generators $z_j$ of $H^* (U(3); \F_p) \cong \Lambda_{\F_p} [z_1, z_3, z_5]$ transgresses to $c_j (\eta)$ in the spectral sequence associated to the fiber sequence \begin{equation} \label{u3t2ss}
    U(3) \longrightarrow U(3)/T(2) \longrightarrow (\CP^{\infty})^2 . 
\end{equation} Due to the multiplicative structure of the spectral sequence, this is enough to find all the differentials. We will deal with the mod $2$ and mod $3$ cases separately from this point forward.
\begin{theo} \label{u3t2mod2}
    The mod $2$ cohomology of $U(3)/T(2)$ is given by \[ H^* (U(3)/T(2) ; \F_2) \cong  \F_2 [z_5 , \beta]/(z_5^2 , \beta^2), \] where $|z_5| = 5$ and $|\beta| =2$.  
\end{theo}
\begin{proof}
    On the second page of the spectral sequence associated with (\ref{u3t2ss}), we have $$ E_2^{*,*} \cong \Lambda_{\F_2} [z_1, z_3,z_5] \otimes \F_2 [\beta_1, \beta_2]. $$
    It follows from the above discussion that on the $E_2$-page, $d_2 (z_1) = 2\beta_1 +\beta_2 = \beta_2$. As $d_3 \equiv 0$, the $E_4 \equiv E_3$-page is thus given by \[ E_4^{*,*} \cong \Lambda_{\F_2} [z_3, z_5] \otimes \frac{\F_2 [\beta_1 , \beta_2]}{(\beta_2)} \cong \Lambda_{\F_2} [z_3, z_5] \otimes \F_2 [\beta_1] . \]
    On the $E_4$ page, $d_4 (z_3) = \beta_1^2$ and hence $E_5 \equiv E_6 $-page is given by \[ E_6^{*,*} \cong \Lambda_{\F_2} [z_5] \otimes \frac{\F_2 [\beta_1]}{(\beta_1^2)} . \] On the $E_6$-page, $d_6 (z_5) = \beta_1^2 \beta_2 = 0$. Since all higher differentials are zero, the spectral sequence collapses at the $E_6$-page. Taking $\beta =\beta_1$, we have the result.
\end{proof}

\begin{theo} \label{u3t2mod3}
    The mod $3$ cohomology of $U(3)/T(2)$ is given by \[ H^* (U(3)/T(2) ; \F_3) \cong  \F_3 [z_3, \beta]/(z_3^2 , \beta^3),\] where $|z_3| = 3$ and $|\beta| =2$.  
\end{theo}
\begin{proof}
    The strategy for this is the same as in the mod $2$ case. The only difference will be the differentials are now given as follows: \begin{align*}
        d_2 (z_1) &= 2\beta_1 + \beta_2 = \beta_2 - \beta_1 \quad \text{on } E_2^{*,*}, \\
        d_4 (z_3) &= \beta_1^2 + 2\beta_1 \beta_2 = 3\beta_1^2 = 0 \quad \quad \text{on } E_4^{*,*}, \\
        d_6 (z_5) &= \beta_1^2 \beta_2 \quad \text{on } E_6^{*,*}. 
    \end{align*}
    Now proceeding similarly as in the proof of Theorem~\ref{u3t2mod2} gives us the desired result.
\end{proof}
\section{Cohomology of \texorpdfstring{$E_{com} U(3)$}{EcomU(3)}} \label{section:2.3}

In this section, we will present a computation for the mod $2$ and mod $3$ cohomologies of $E_{com} U(3)$. To do this, we will utilize the description of $E_{com} U(3)$ from Theorem~\ref{adem-gomez-hocolim} and the spectral sequence associated with a homotopy colimit. But first, let us recall some results on the spectral sequence associated with a homotopy colimit and higher limit functors.

\subsection{Spectral Sequence associated with a Homotopy Colimit} \label{subsection:2.3.1}

\begin{dfn} \label{co-sim-rep}
Let $\mathbf{A}$ be an abelian category, $I$ be a small category and $D: I \rightarrow \mathbf{A}$ be a functor. We write the \textit{cosimplicial replacement} for $D$: \[
\prod_\ii D(\ii) \doublerightarrow{}{} \prod_{\ii_0 \rightarrow \ii_1} D(\ii_1) \triplerightarrow{}{} \prod_{\ii_0 \rightarrow \ii_1 \rightarrow \ii_2} D(\ii_2) \quadrightarrow{}{} \cdots     \]
This is a cosimplicial object over $\mathbf{A}$. The alternating sum of the coface maps defines differentials, which defines a cochain complex over $\mathbf{A}$. We define $H^d (I; D)$ to be the $d$-th cohomology of this cochain complex.  
\end{dfn}

Note that $H^0(I; D)$, which is the equalizer of the first two arrows, can be identified as the ordinary limit $\mathrm{lim} D$. This means the groups $H^d (I; D)$ can be interpreted as ``higher limit functors'' $ \mathrm{lim}^d D$. For our purposes, the abelian category $\mathbf{A}$ will be the category of abelian groups $\mathbf{Ab}$ and $I$ will be the poset $\mathcal{S} (2)$. Moreover, $D$ is equal to $H^k (\mathcal{H}; \F)$ where $\F$ is either $\F_3$ or $\F_2$. So, we have a diagram similar to Figure~\ref{fig:poset2} but with the arrows reversed (see Figure~\ref{fig:opposet}). This gives the following functor: \begin{align*}
    H^k (\mathcal{H} ; \F) : \mathcal{S} (2)^{op} &\longrightarrow \mathbf{Ab} \\ 
    \ii &\longmapsto H^k (\mathcal{H} (\ii) ; \F).
\end{align*} 
\begin{figure}
    \centering
    \begin{tikzcd}
    & H^k (\mathcal{H} (1,2)) \ar[d] \\ H^k (\mathcal{H} (2)) \ar[ur] \ar[d] \ar[r] & H^k    (\mathcal{H} (0,1,2))  & H^k (\mathcal{H} (1)) \ar[d] \ar[ul] \ar[l]  \\ H^k       (\mathcal{H} (0,2)) \ar[ur]  &  H^k (\mathcal{H}   (0)) \ar[u] \ar[r] \ar[l] & H^k (\mathcal{H} (0,1)) \ar[ul] 
    \end{tikzcd}
    \caption{Diagram of $H^k (\mathcal{H} (\ii))$ for the poset $\mathcal{S} (2)$}
    \label{fig:opposet}
\end{figure}

To compute the higher limits, we write the cosimplicial replacement of the diagram in Figure~\ref{fig:opposet}. By Definition~\ref{co-sim-rep}, the cosimplicial replacement of $H^k (\mathcal{H} ; \F)$ is given by \begin{equation} \label{cosimrepS2}
      \prod_{\ii \in\mathcal{S} (2)} H^k (\mathcal{H} (\ii) ;\F) \doublerightarrow{}{} \prod_{\ii_0 \rightarrow \ii_1} H^k (\mathcal{H} (\ii_1) ;\F) \triplerightarrow{}{} \prod_{\ii_0 \rightarrow \ii_1 \rightarrow \ii_2} H^k (\mathcal{H} (\ii_2) ;\F) \longrightarrow 0 \end{equation}    
such that the projection of the coface map $d^j : \prod_{\ii} H^k (\mathcal{H} (\ii) ;\F) \rightarrow \prod_{\ii_0 \rightarrow \ii_1} H^k (\mathcal{H} (\ii_1) ;\F)$ onto the factor $H^k (\mathcal{H} (\ii_1);\F)$ indexed by $(\ii_0 \rightarrow \ii_1)$ is the composition of a projection from the product with \begin{itemize}
    \item the identity map from the factor $H^k (\mathcal{H} (\ii_1) ; \F)$ when $j=0$,
    \item the map $H^k (\mathcal{H} (\ii_0) ; \F) \rightarrow H^k (\mathcal{H} (\ii_1) ; \F)$ when $j=1$. 
\end{itemize}
The projection of the coface map $d^j : \prod_{\ii_0 \rightarrow \ii_1} H^k (\mathcal{H} (\ii_1) ;\F) \rightarrow \prod_{\ii_0 \rightarrow \ii_1 \rightarrow \ii_2} H^k (\mathcal{H} (\ii_2) ;\F)$ onto the factor $H^k (\mathcal{H} (\ii_2);\F)$ indexed by $(\ii_0 \rightarrow \ii_1 \rightarrow \ii_2)$ is the composition of a projection from the product with \begin{itemize}
    \item the identity map from the factor $H^k (\mathcal{H} (\ii_2) ;\F)$ indexed by $\ii_1 \rightarrow \ii_2$ when $j=0$,
    \item the identity map from the factor $H^k (\mathcal{H} (\ii_2) ;\F)$ indexed by $\ii_0 \rightarrow \ii_2$ when $j=1$,
    \item the map $H^k (\mathcal{H} (\ii_1) ; \F) \rightarrow H^k (\mathcal{H} (\ii_2) ; \F)$ indexed by $\ii_0 \rightarrow \ii_1$ when $j=2$.
\end{itemize}
Taking the alternate sum of the coface maps $d^j$ makes (\ref{cosimrepS2}) a cochain complex. The following theorem is due to Bausfield-Kan \cite{bk72}.

\begin{theo} \label{ssbauskan}
Let $D: I \rightarrow \mathbf{Top} $ be a diagram and let $\mathcal{E}_*$ be a homology theory. Write $\mathcal{E}^k (D)$ for the diagram $I^{op} \rightarrow \mathbf{Ab}$. Then there is a spectral sequence $E^{d,k}_2 = H^d (I^{op} ; \mathcal{E}^k (D) ) $ converging to $\mathcal{E}^{d+k} (\mathrm{hocolim} D)$. The differentials have the form $d^r : E^{d,k}_r \longrightarrow E^{d+r, k-r+1}_r $.
\end{theo}

To simplify our computations for the category $\mathcal{S} (2)$ and the functor $H^k (\mathcal{H}; \F)$, we can limit our focus to cohomologies $H^d (\mathcal{S} (2)^{op}; H^k (\mathcal{H} ;\F))$ for $d \le 2$ since the diagram in Figure~\ref{fig:opposet} only contains chains of length up to $3$, resulting in non-zero terms in (\ref{cosimrepS2}) only up to codimension $3$. Moreover, according to the following lemma, $H^2 (\mathcal{S} (2)^{op}; H^k (\mathcal{H};\F))$ is always zero, thus we only need to compute cohomologies for $d=0,1$.  

\begin{lema} \label{lem:H2(S2;H)}
Let $\F$ be a field. Then $\mathrm{lim}^2 H^k (\mathcal{H}; \F) =H^2 (\mathcal{S} (2)^{op} ; H^k (\mathcal{H}; \F) )  \cong 0$.    
\end{lema}
\begin{proof}
From the description of the spaces $\mathcal{H} (\ii)$ in \S \ref{section:2.2}, we see that $\mathcal{H} (0,2) = \mathcal{H} (0,1,2)$. Consider the map \[ d^0-d^1+d^2 : \prod_{\ii_0 \rightarrow \ii_1} H^k (\mathcal{H} (\ii_1) ;\F) \longrightarrow \prod_{\ii_0 \rightarrow \ii_1 \rightarrow \ii_2} H^k (\mathcal{H} (\ii_2) ;\F). \] Note that in $\mathcal{S} (2)^{op}$ all $3$-chains end with $(0,1,2)$ and hence $\ii_2 = (0,1,2)$ is constant in all terms in the product $\prod_{\ii_0 \rightarrow \ii_1 \rightarrow \ii_2} H^k (\mathcal{H} (\ii_2) ;\F)$. So, we can rewrite above map as \[ d^0-d^1+d^2 : \prod_{\ii_0 \rightarrow \ii_1} H^k (\mathcal{H} (\ii_1) ;\F) \longrightarrow \prod_{\ii_0 \rightarrow \ii_1 \rightarrow (0,1,2)} H^k (\mathcal{H} (0,1,2) ;\F). \] Let any $(a_1, \dots ,a_6) \in \prod_{\ii_0 \rightarrow \ii_1 \rightarrow (0,1,2)} H^k (\mathcal{H} (0,1,2) ;\F)$ be given. We pick an element $b$ in $\prod_{\ii_0 \rightarrow \ii_1} H^k (\mathcal{H} (\ii_1) ;\F)$ as follows. Choose $a_3 -a_1, a_4 -a_6 , a_3, a_4 \in H^k (\mathcal{H} (0,1,2) ;\F)$ as entries of $b$ that are indexed by $(0)\rightarrow (0,1,2)$, $(2) \rightarrow (0,1,2)$, $(0,1) \rightarrow (0,1,2)$, $(1,2)\rightarrow (0,1,2)$ respectively. For the entries which are indexed by $(0) \rightarrow (0,2)$, $(2) \rightarrow (0,1,2)$ choose $a_2 +a_3 -a_1,a_5 +a_4 -a_6 \in H^k (\mathcal{H} (0,2) ; \F) = H^k (\mathcal{H} (0,1,2) ;\F)$ respectively. Choose all other entries in $b$ to be $0$. Then from our description of the cochain maps, \[ d^0 -d^1 +d^2 : b \longmapsto (a_1 ,\dots,a_6 ) \in \prod_{\ii_0 \rightarrow \ii_1 \rightarrow (0,1,2)} H^k (\mathcal{H} (0,1,2) ;\F). \] This shows that $d^0 - d^1 +d^2$ is surjective and hence $H^2 (\mathcal{S} (2)^{op} ; H^k (\mathcal{H}; R) )  = 0$.     
\end{proof}  

\subsection{Mod \texorpdfstring{$p$}{p} Cohomology of \texorpdfstring{$E_{com} U(3)$}{EcomU(3)}} \label{subsection:2.3.2}

First we show that $H^* (E_{com} U(3) ; \Z )$ can have only $2$ and $3$-torsion. This follows from a much more general statement that the integral cohomology of $B_{com} G_{\mathds{1}}$ and $E_{com} G_{\mathds{1}}$ can have only $p$-torsion for $ p$ dividing the order of the Weyl group $W$ of $G$. To prove this we need the following lemma.

\begin{lema} \label{lema:iso-F} (\cite{Adem-Gomez})
Fix a compact connected Lie group $G$. Let $T \subset G$ be a maximal torus and $W$ be the Weyl group. Then we have isomorphisms \[ H^* (G/T \times_{W} BT ; \F ) \cong H^* (B_{com} G_{\mathds{1}} ; \F) \quad \text{and} \quad H^* (G/T \times_{W} G/T ; \F ) \cong H^* (E_{com} G_{\mathds{1}} ; \F) \] whenever $\mathrm{gcd} (\mathrm{char} (\F) , |W|) = 1$.
\end{lema}

\begin{prop} \label{prop:ptorsion}
Let $G$ be a compact connected Lie group, $T \subset G$ be a maximal torus and $W$ be the corresponding Weyl group. Then $H^* (B_{com} G_{\mathds{1}} ; \Z)$ and $H^* (E_{com} G_{\mathds{1}} ; \Z)$ can only have $p$-torsion for those $p$ dividing $|W|$.
\end{prop}
\begin{proof}
From Lemma~\ref{lema:iso-F}, it is enough to prove the following claim.\\
\textbf{Claim.} $H^* (G/T \times_W BT ; \Z)$ and $H^* (G/T \times_W G/T ; \Z )$ can only have $p$-torsion for those $p$ dividing $|W|$. \\
\textit{Proof of Claim.} Consider the covering space $\pi: G/T \times BT \rightarrow G/T \times_W BT$ and the transfer homomorphism $\tau^* : H^k (G/T \times BT ;R) \rightarrow H^k (G/T \times_W BT ; R)$, where $R$ is a commutative ring. The composision $\tau^* \circ \pi^* $ is the multiplication by $|W|$ in the $R$-module $H^k ( G/T\times_W BT ; R)$. \\
Let $q$ be a prime such that $\mathrm{gcd} (q, |W|) =1$. Set $R = \Z_{q} $ the localization of the integers at $q$. Since $\mathrm{gcd} (q, |W|) =1$, $|W|$ is invertible in $\Z_{q}$ and multiplication by $|W|$ in $H^k ( G/T\times_W BT ; \Z_{q})$ is an isomorphism. Therefore $\tau^* \circ \pi^*$ is an isomorphism and $\pi^* : H^k (G/T \times_W BT ;\Z_{q}) \rightarrow H^k (G/T \times BT ; \Z_{q})$ is injective. As $H^* (G/T \times BT ; \Z)$ is torsion free, $H^* (G/T \times_W BT ; \Z)$ has no $q$-torsion, otherwise $\pi^*$ would not be injective. Therefore $H^* (G/T \times_W BT ; \Z)$ can only have $p$-torsion for those $p$ dividing $|W|$. The argument for $H^* (G/T \times_W G/T ; \Z )$ is similar.
\end{proof}

Due to Proposition~\ref{prop:ptorsion}, $H^* (E_{com} U(3) ; \Z)$ only has torsion at the primes 2 and 3. So, the integral cohomology of $E_{com} U(3)$ can be completely described from the mod $2$ and mod $3$ cohomologies. We focus on computing these in the rest of this section. \\

With $\F_2$ coefficients, we can describe the $E_2$-page of the spectral sequence in Theorem~\ref{ssbauskan}. Only finitely many terms in the spectral sequence are non-zero. Explicitly from the computation of $\mathrm{lim}^0$ and $\mathrm{lim}^1$ (see Appendix~\ref{appendixA.2}), we have \begin{align*}
    E_2^{0,k} = \mathrm{lim}^0 H^k (\mathcal{H} ; \F_2) &\cong \F_2 \quad \text{for } k=0,4,5\\ 
    E_2^{1,k} = \mathrm{lim}^1 H^k (\mathcal{H} ; \F_2) &\cong \F_2 \quad \text{for } k=3,5,8,11,12,13\\
    E_2^{0,6} = \mathrm{lim}^0 H^6 (\mathcal{H} ; \F_2) &\cong \F_2^{\oplus 2} \\
    E_2^{1,7} = \mathrm{lim}^1 H^7 (\mathcal{H} ; \F_2) &\cong \F_2^{\oplus 2}  \\
    E_2^{d,k} = \mathrm{lim}^d H^k (\mathcal{H} ; \F_2) &\cong 0  \quad \text{otherwise}.
\end{align*}

We see immediately that there is no non-trivial differential on the $E_2$-page. Hence the spectral sequence collapses at the $E_2$-page and $E_2^{*,*} = E_{\infty}^{*,*}$. This gives the following theorem: 
\begin{theo} \label{EcomU(3)-mod2}
The $\F_2$-cohomology of $E_{com} U(3)$ is given by \[ H^d (E_{com} U(3) ; \F_2 ) = \begin{cases} \F_2 & d=0,5,9,12,13,14 \\ \F_2^{\oplus 2} & d=4,8 \\ \F_2^{\oplus 3} & d=6 \\ 0  & \text{otherwise}   \end{cases} . \] 
\end{theo}

Similarly, with $\F_3$ coefficients, we can also describe the $E_2$-page of the spectral sequence in Theorem~\ref{ssbauskan}. Again only finitely many terms in the spectral sequence are non-zero. Explicitly from the computation of $\mathrm{lim}^0$ and $\mathrm{lim}^1$ as before, we have \begin{align*}
    E_2^{0,k} = \mathrm{lim}^0 H^k (\mathcal{H} ; \F_3) &\cong \F_3  \quad \text{for } k=0,4,5,6\\ 
    E_2^{1,k} = \mathrm{lim}^1 H^k (\mathcal{H} ; \F_3) &\cong \F_3 \quad \text{for } k=3,6,12\\
    E_2^{1,k} = \mathrm{lim}^1 H^k (\mathcal{H} ; \F_3) &\cong \F_3^{\oplus 2}  \quad \text{for } k=5,8,11 \\
    E_2^{1,k} = \mathrm{lim}^1 H^k (\mathcal{H} ; \F_3) &\cong \F_3^{\oplus 3} \quad \text{for } k= 7,9,10 \\
    E_2^{d,k} = \mathrm{lim}^d H^k (\mathcal{H} ; \F_3) &\cong 0  \quad \text{otherwise}.
\end{align*}

As there are no non-trivial differentials on the $E_2$-page, the spectral sequence collapses and $E_2^{*,*} = E_{\infty}^{*,*}$. This gives the following theorem: 
\begin{theo} \label{EcomU(3)-mod3}
The $\F_3$-cohomology of $E_{com} U(3)$ is given by \[ H^d (E_{com} U(3) ; \F_3 ) = \begin{cases} \F_3 & d=0,5,7,13 \\ \F_3^{\oplus 2} & d=4,9,12 \\ \F_3^{\oplus 3} & d=6,8,10,11 \\ 0 & \text{otherwise}   \end{cases} . \] 
\end{theo}

\subsection{Rational Cohomology Ring of \texorpdfstring{$E_{com} U(3)$}{EcomU(3)}} \label{section:2.3.3}

We finish this section by describing the rational cohomology ring of $E_{com} U(3)$. For the rest of this section, all cohomologies are taken with rational coefficients. From (\cite{Adem-Gomez}, Corollary 7.4) we have an isomorphism \begin{equation} \label{eq19}
     H^* (E_{com} U(3) ) \cong \big ( H^* (\Fl_3 (\C)) \otimes H^* (\Fl_3 (\C)) \big )^{\Sigma_3}. \end{equation}
Recall that $H^* (\Fl_3 (\C)) \cong \Q [x_1 , x_2 ,x_3]/(\sigma_1 , \sigma_2, \sigma_3) $, where $\sigma_i$ is the $i$-the elementary symmetric polynomial in the $x_i$'s. The averaging operator is defined as follows: \begin{align*}
      \rho : H^* (\Fl_3 (\C) ) \otimes H^* (\Fl_3 (\C) ) &\longrightarrow \big ( H^* (\Fl_3 (\C) ) \otimes H^* (\Fl_3 (\C) ) \big )^{\Sigma_3} \\
      f(\mathbf{x},\mathbf{y}) &\longmapsto  \frac{1}{6} \sum_{\omega \in \Sigma_3} f(\omega \mathbf{x} , \omega \mathbf{y}). 
      \end{align*}
For each $\omega \in \Sigma_3$ the diagonal descent monomial is defined as follows: \begin{equation}
     f_{\omega} := \prod_{\omega^{-1} (i) > \omega^{-1} (i+1)} (x_1 \cdots x_i) \otimes \prod_{\omega (j) > \omega (j+1)} (y_{\omega (1)} \cdots y_{\omega (j)} ). \end{equation} 
By (\cite{Allen1994}, Theorem~1.3) the collection $\{ \rho (f_{\omega}) \}_{\omega \in \Sigma_3}$ forms a free basis of $(H^* (\Fl_3 (\C) ) \otimes H^* (\Fl_3 (\C) ) )^{\Sigma_3}$ as a $\Q$-module. So, the set 
\begin{equation*}
    \mathcal{B} = \{ 1 , \rho (x_1 \otimes y_2) , \rho (x_1 \otimes y_2 y_3) , \rho (x_1 x_2 \otimes y_3) , \rho (x_1 x_2 \otimes y_2 y_3) , \rho (x_1^2 x_2 \otimes y_3^2 y_2) \} 
\end{equation*}
is basis for $(H^* (\Fl_3 (\C) ) \otimes H^* (\Fl_3 (\C) ) )^{\Sigma_3}$. It follows that images of $\rho (f_{\omega} )$ for $\omega \in \Sigma_3$ under the isomorphism (\ref{eq19}) form a basis of $H^* (E_{com} U(3) )$. With a slight abuse of notation, we also denote the image as $\rho (f_{\omega})$ for now. Also note that the degree of $\rho (f_{\omega} )$ is given by $\mathrm{deg} (\rho (f_{\omega})) = 2(\mathrm{maj} (\omega) + \mathrm{maj} (\omega^{-1}) )$ where \[ \mathrm{maj} (\omega) := \sum_{\omega (i) > \omega (i+1)} i.  \]  
\begin{theo}
\label{theo5.11}
We have the following relations between the classes in $\mathcal{B}$:
\begin{enumerate}
    \item $2 \rho (x_1 \otimes y_2)^2 = -\rho (x_1 x_2 \otimes y_2 y_3)$,
    \item $ 3\rho (x_1 \otimes y_2 y_3) \cdot \rho (x_1 x_2 \otimes y_3) = 2\rho (x_1^2 x_2 \otimes y_3^2 y_2)$,
    \item All other products are zero.
\end{enumerate}
\end{theo}
\begin{proof}
We will show the relations in $\big ( H^* (\Fl_3 (\C)) \otimes H^* (\Fl_3 (\C) \big )^{\Sigma_3}$, which is enough due to the isomorphism (\ref{eq19}).
First note that in $H^* (\Fl_3 (\C))$ we have the following relation between the generators: $x_i^2 = x_j x_k$, $i\neq j \neq k$. This follows from combining the relations $\sigma_1 = \sigma_2 = 0$. Using the definition of $\rho$ given above we check that \begin{align*}
    \rho (x_1 \otimes y_2) &= -\frac{1}{6} (x_1 \otimes y_1 + x_2 \otimes y_2 + x_3 \otimes y_3)\\
    \rho (x_1 x_2 \otimes y_2 y_3) &= -\frac{1}{6} ( x_1 x_2 \otimes y_1 y_2 + x_2 x_3 \otimes y_2 y_3 +x_1 x_3 \otimes y_1 y_3) 
\end{align*}    
Hence, \begin{align*}
    2 \rho (x_1 \otimes y_2)^2 &= \frac{1}{18} (x_1^2 \otimes y_1^2 + x_2^2 \otimes y_2^2 + x_3^2 \otimes y_3^2 + 2 x_1 x_2 \otimes y_1 y_2  \\ & \hspace{5cm} + 2 x_2 x_3 \otimes y_2 y_3 + 2 x_1 x_3 \otimes y_1 y_3) \\
    & = \frac{1}{6} (x_1 x_2 \otimes y_1 y_2 + x_2 x_3 \otimes y_2 y_3 +x_1 x_3 \otimes y_1 y_3 )  \quad \quad (\text{using } x_i^2 = x_j x_k) \\
    &= - \rho (x_1 x_2 \otimes y_2 y_3)
\end{align*}
The other relations also follow directly by writing down the corresponding $\rho (f_{\omega})$ and using the relations $\sigma_1 = \sigma_2 = \sigma_3 = 0$.
\end{proof}
\begin{cor}
Let us denote image of $\rho (x_1 \otimes y_2)$ as $\gamma_4$, $\rho (x_1 \otimes y_2 y_3)$ as $\gamma_6$, $\rho (x_1 x_2 \otimes y_3)$ as $\widetilde{\gamma_6}$ in $H^* (E_{com} U(3) ; \Q)$ under the isomorphism (\ref{eq19}). Then \[  H^* (E_{com} U(3) ; \Q) \cong \frac{\Q [\gamma_4 , \gamma_6 , \widetilde{\gamma_6} ]}{(\gamma_4^3 , \gamma_6^2 , \widetilde{\gamma}_6^2 , \gamma_4 \gamma_6 , \gamma_4 \widetilde{\gamma_6})} . \]
\end{cor}

\appendix
\section{Appendix}

\subsection{Computation using GAP for \S \ref{subsection:2.2.2}} \label{appendixA.1}

For the computations of the group cohomologies in Proposition~\ref{3MMSM}, we used the HAP \cite{HAP} package in the computer algebra system GAP \cite{GAP4}. The HAP package is commonly used for homological algebra computations. We want to compute the group cohomology of $\Sigma_3$ with coefficients in the module $M^{\otimes 2}$ and $M_S = M\otimes S$, where $M$ and $S$ are the standard and sign representation of $\Sigma_3$ respectively. Also recall, the $\Z \Sigma_3$-modules $M$ and $S$ are given by \[ M = \frac{\Z \{ x_1, x_2, x_3 \}}{(x_1 + x_2 + x_3)} \quad \text{and} \quad S= \Z \{t\},  \]  The actions of $\sigma$ and $\tau$ on the generators of the modules $M$ and $S$ are as follows:  \begin{align*}
    \tau \cdot x_1  = x_2 \quad &  \sigma \cdot x_1 = x_2  \\
    \tau \cdot x_2 = x_1 \quad & \sigma \cdot x_2 = - x_1 - x_2  \\ 
    \tau \cdot t = -t \quad  & \sigma \cdot t = t .
\end{align*}
The underlying $\Z$-module of $M^{\otimes 2}$ and $ M_S$ are isomorphic to \begin{align*} 
  \Z \{ x_1 \otimes x_1, x_1 \otimes x_2, x_2 \otimes x_1, x_2 \otimes x_2 \} , \quad \text{and} \quad \Z \{ x_1 \otimes t, x_2 \otimes t \}
\end{align*} 
respectively. The actions of $\sigma$ and $\tau$ on the generators of $M^{\otimes 2}$ and $M_S$ are induced by the actions on $x_1$, $x_2$, and $t$. In GAP, we first define the symmetric group $G:=\Sigma_3$, the general linear group $H:= GL(n, \Z)$, and the generators \verb|gens| of $G$ where $\sigma := (123)$ and $\tau = (12)$. A group homomorphism $A: G \rightarrow GL(n, \Z)$ can be viewed as the $\Z G$-module with the underlying abelian group $\Z^n$ on which $G$ acts via the homomorphism. To define a group homomorphism in GAP we use the function \verb|GroupHomomorphismByImages|, which takes 4 inputs, namely the domain and codomain groups, the generators of the domain, and the images of the generators in the codomain. The matrices $\sigma$ and $\tau$ map to in $GL(n, \Z)$ are denoted as \verb|sigma| and \verb|tau| respectively. The module $M_S$ is given by the following group homomorphism: \begin{align*}
    \Sigma_3 &\longrightarrow GL(2, \Z) \\
    \sigma &\mapsto \begin{pmatrix} 0 & -1 \\  1 & -1  \end{pmatrix} := m_{\sigma}, \quad \tau \mapsto \begin{pmatrix}
         0 & -1 \\ -1 & 0  
    \end{pmatrix} := m_{\tau}
\end{align*}
and the module $M^{\otimes 2}$ is given by the following group homomorphism: \begin{align*}
    \Sigma_3 &\longrightarrow GL(4, \Z) \\
    \sigma &\mapsto \begin{pmatrix} 0 & - m_{\sigma} \\ m_{\sigma} & - m_{\sigma} &  \end{pmatrix}, \quad \tau \mapsto \begin{pmatrix}
         0 & -m_{\tau} & \\ -m_{\tau} & 0 
    \end{pmatrix}.
\end{align*} 
Recall that $H^d (G; A) = \mathrm{Ext}_{\Z G}^d (\Z , A)$. The function \verb|ResolutionFiniteGroup(G, k)| takes a finite group $G$ and an integer $k\ge 1$ and returns $k+1$ terms of a free $\Z G$-resolution of $\Z$. The function \verb|HomToIntegralModule(R, A)| inputs a free $\Z G$-resolution $R$ and a group homomorphism $A: G \rightarrow GL(n, \Z)$ (viewed as a $\Z G$-module) and returns the cochain complex $\mathrm{Hom}_{\Z G} (R, A)$. Finally the function \verb|Cohomology(C, d)| returns the $d$-th cohomology of $G$ with coefficients in the $\Z G$-module $A$. The following code prints $H^d (\Sigma_3 ; M_S )$ for $0\le d \le 11$ as a list.
\begin{lstlisting}[language=GAP]
LoadPackage("HAP");
G := SymmetricGroup(G);;
H := GL(2, Integers);;
gens := GeneratorsOfGroup(G);;
sigma := [[0,-1],[1,-1]];;
tau := [[0,-1],[-1,0]];;
A := GroupHomomorphismByImages(G, H, gens, [sigma, tau]);;
R := ResolutionFiniteGroup(G, 12);;
C := HomToIntegralModule(R, A);;
List([0..11], d -> [d, Cohomology(C, d)]);
\end{lstlisting}
The obvious modifications to the above code produces $H^d (\Sigma_3; M^{\otimes 2})$:
\begin{lstlisting}[language=GAP]
LoadPackage("HAP");
G := SymmetricGroup(G);;
H := GL(4, Integers);;
gens := GeneratorsOfGroup(G);;
sigma := [[0,0,0,1],[0,0,-1,1],[0,-1,0,1],[1,-1,-1,1]];;
tau := [[0,0,0,1],[0,0,1,0],[0,1,0,0],[1,0,0,0]];;
A := GroupHomomorphismByImages(G, H, gens, [sigma, tau]);;
R := ResolutionFiniteGroup(G, 12);;
C := HomToIntegralModule(R, A);;
List([0..11], d -> [d, Cohomology(C, d)]);
\end{lstlisting}
Note that, we can only compute $H^d (\Sigma_3; M^{\otimes 2})$ and $H^d (\Sigma_3 ; M_S)$ up to a finite $d$, but both $H^* (\Sigma_3 ; M^{\otimes 2})$ and $H^* (\Sigma_3 ; M_S)$ are infinite. This issue is resolved by noticing that the group cohomology of $\Sigma_3$ is periodic.

\subsection{Select Computations of \texorpdfstring{$\mathrm{lim}^0$}{lim0} and \texorpdfstring{$\mathrm{lim}^1$}{lim1} in \S \ref{subsection:2.3.2}} \label{appendixA.2}

Let $\ii_0 = (0)$, $\ii_1 = (1)$, $\ii_2 = (2)$, $\ii_3 = (0,1)$, $\ii_4 = (0,2)$, $\ii_5 = (1,2)$, and $\ii_6 = (0,1,2)$. Recall from \S \ref{subsection:2.3.1} that $H^d (\mathcal{S} (2)^{op} ; H^k (\mathcal{H}; \F)) $ is equal to the $d$-th cohomology of the following chain complex.
\begin{equation} \label{chaincomplexH}
    \begin{tikzcd}
    H^k (\mathcal{H} (\ii_0)) \times H^k (\mathcal{H} (\ii_1)) \times H^k (\mathcal{H} (\ii_2)) \times H^k (\mathcal{H} (\ii_3)) \times H^k (\mathcal{H} (\ii_4)) \times H^k (\mathcal{H} (\ii_5)) \times H^k (\mathcal{H} (\ii_6)) \ar[d, "d_0 -d_1"] \\
    H^k (\mathcal{H} (\ii_3)^{\times 2}) \times H^k (\mathcal{H} (\ii_4)^{\times 2}) \times H^k (\mathcal{H} (\ii_5)^{\times 2}) \times H^k (\mathcal{H} (\ii_6)^{\times 6}) \ar[d, "d_0 - d_1 + d_2"] \\
    H^k (\mathcal{H} (\ii_6)^{\times 6}) \ar[d] \\
    0
\end{tikzcd}
\end{equation} 
Let us also denote the map induced by $\ii_j \rightarrow \ii_l$ as $p_{jl} : H^k (\mathcal{H} (\ii_j)) \rightarrow H^k (\mathcal{H} (\ii_l))$. From our computations from \S \ref{section:2.2}, \cite{BAUM1965}, and classical results we can describe the mod $2$ Poincar{\'e} series of all the spaces $\mathcal{\ii}$ for $\ii \in \mathcal{S} (2)$: \begin{align*}
    \Pi_{\mathcal{H}(\ii_0)} (t) &= 1+t^3+t^5+t^8 \\
    \Pi_{\mathcal{H}(\ii_1)} (t) &= 1+2t^2+2t^4+t^5+t^6+2t^7+2t^9+t^{11} \\
    \Pi_{\mathcal{H}(\ii_2)} (t) &= 1+t+t^2+t^3+2t^4+t^5+4t^6+t^7+2t^8+t^9+t^{10}+t^{11}+t^{12} \\
    \Pi_{\mathcal{H}(\ii_3)} (t) &= 1+t^2+t^3+t^4+2t^5+2t^7+t^8+t^9+t^{10}+t^{12} \\
    \Pi_{\mathcal{H}(\ii_4)} (t) &= 1+t+t^2+2t^3+2t^4+3t^5+3t^6+2t^7+3t^8+3t^9+2t^{10}+2t^{11}+t^{12}+t^{13}+t^{14} \\
    \Pi_{\mathcal{H}(\ii_5)} (t) &= 1+t+2t^2+2t^3+2t^4+3t^5+3t^6+3t^7+3t^8+2t^9+2t^{10}+2t^{11}+t^{12}+t^{13} \\
    \Pi_{\mathcal{H}(\ii_6)} (t) &= 1+t+t^2+2t^3+2t^4+3t^5+3t^6+2t^7+3t^8+3t^9+2t^{10}+2t^{11}+t^{12}+t^{13}+t^{14} 
\end{align*} 
From Proposition~\ref{lem:H2(S2;H)}, $d^0 -d^1 +d^2$ is surjective. So, we only need to know $\mathrm{ker} (d^0 - d^1) $ to compute $\lim^0$ and $\lim^1$. From the description of all the mod $2$ Poincar{\'e} series of the spaces $\mathcal{H} (\ii)$ above, we will have possible non-zero $\lim^0 H^k (\mathcal{H}; \F_2)$ and $\lim^1 H^k (\mathcal{H} ; \F_2)$ for $0\le k \le 14$. We will present details for a few of them. \\

Let us start with the case $k=0$. The coface maps $d^0$ and $d^1$ are given as follows: \begin{align*}
    d^0 : (a_0, a_1, a_2, a_3, a_4, a_5, a_6) & \longmapsto (a_3, a_3, a_4, a_4, a_5, a_5, a_6, \dots a_6), \\
    d^1 : (a_0, a_1, a_2, a_3, a_4, a_5, a_6) & \longmapsto (p_{03} (a_0), p_{13} (a_1), p_{04} (a_0) , p_{24} (a_2), p_{15} (a_1), p_{25} (a_2), \\
                       & \hspace{6cm} p_{06} (a_0), \dots , p_{56} (a_5) ).
\end{align*}
The kernel of $d^0 -d^1$ is the equalizer of $d^0$ and $d^1$. We thus have \begin{align*}
    \mathrm{ker} (d^0 -d^1) &= \{ (a_0, a_1, a_2, a_3, a_4, a_5, a_6) \mid  a_3 = p_{j3} (a_j), a_4 = p_{l4} (a_l), a_5 = p_{m5} (a_m), \\
     & \hspace{3cm} a_6 = p_{n6} (a_n) \text{ for } j=0,1, l=0,2, m=1,2, \text{ and } 0\le n\le 5 \} \\
     &= \{ (a_0, a_1, a_2, a_3, a_4, a_5, a_6) \mid a_0 = a_1 = a_2 = a_3 = a_4 = a_5 = a_6 \}\\
     &\cong \F_2
\end{align*}
The chain complex (\ref{chaincomplexH}) in this case is $\F_2^{\oplus 7} \rightarrow \F_2^{\oplus 12} \rightarrow \F_2^{\oplus 6} \rightarrow 0$.
Therefore $$\mathrm{lim}^0 H^0 (\mathcal{H} ; \F_2) \cong \F_2  \text{ and } \mathrm{lim}^1 H^0 (\mathcal{H} ; \F_2) = 0 . $$ \\
The next non-trivial example is $k=3$. The coface maps are as follows: \begin{align*}
    d^0 : (a_0, a_2, a_3, a_4, a_5, a_6) & \longmapsto (a_3, a_3, a_4, a_4, a_5, a_5, a_6, \dots a_6), \\
    d^1 : (a_0, a_2, a_3, a_4, a_5, a_6) & \longmapsto (p_{03} (a_0), 0, p_{04} (a_0) , p_{24} (a_2), 0, p_{25} (a_2), p_{06} (a_0), \dots , p_{56} (a_5) ).
\end{align*}
The kernel of $d^0 - d^1$ is given by \begin{align*}
    \mathrm{ker} (d^0 -d^1) &= \{ (a_0, a_2, a_3, a_4, a_5, a_6) \mid  a_3 = a_4 = a_5 = a_6 = 0, p_{03} (a_0) = p_{04} (a_0) = 0, \\
    & \hspace{6cm} p_{24} (a_2) = p_{25} (a_2) = 0 \} \\
     &\cong 0
\end{align*}
as the maps $p_{03}$, $ p_{04}$, $ p_{24}$, and $p_{25}$ are all injective. The chain complex (\ref{chaincomplexH}) in this case is $\F_2^{\oplus 9} \rightarrow \F_2^{\oplus 22} \rightarrow \F_2^{\oplus 12} \rightarrow 0$.
Therefore $$\mathrm{lim}^0 H^0 (\mathcal{H} ; \F_2) \cong 0  \text{ and } \mathrm{lim}^1 H^0 (\mathcal{H} ; \F_2) = \F_2 . $$\\
Let us see one more example, $k=4$. The coface maps are as follows \begin{align*}
    d^0 : (a_1, a_2, a_3, a_4, a_5, a_6) & \longmapsto (a_3, a_3, a_4, a_4, a_5, a_5, a_6, \dots a_6), \\
    d^1 : (a_1, a_2, a_3, a_4, a_5, a_6) & \longmapsto (0, p_{13} (a_1), 0 , p_{24} (a_2), p_{15} (a_1), p_{25} (a_2), 0, p_{16} (a_1), \dots , p_{56} (a_5) ).
\end{align*}
The kernel of $d^0 - d^1$ is given by \begin{align*}
    \mathrm{ker} (d^0 -d^1) &= \{ (a_1, a_2, a_3, a_4, a_5, a_6) \mid  a_3 = a_4 = a_5 = a_6 = 0, p_{24} (a_2) = 0, p_{13} (a_1) = 0 \} \\
     &= \{ (a_1^1, a_1^2) \mid p_{13} (a_1^1, a_1^2) = 0 \}\\
     &\cong \F_2
\end{align*}
as the map $p_{13} : \mathcal{H}(\ii_1) \cong \F_2^{\oplus 2} \rightarrow \mathcal{H}(\ii_3) \cong \F_2$ is given by projection of the first coordinate. The chain complex (\ref{chaincomplexH}) in this case is $\F_2^{\oplus 11} \rightarrow \F_2^{\oplus 22} \rightarrow \F_2^{\oplus 12} \rightarrow 0 $. Therefore $$\mathrm{lim}^0 H^4 (\mathcal{H} ; \F_2) \cong \F_2  \text{ and } \mathrm{lim}^1 H^4 (\mathcal{H} ; \F_2) = 0 . $$

\bibliographystyle{plainnat} 
\bibliography{refs} 

\end{document}